\documentclass[10pt,a4paper]{article}
\usepackage{amsmath}
\usepackage{amssymb,bm}
\usepackage{amsbsy}
\usepackage{amsfonts}
\usepackage{graphicx}
\usepackage{sectsty}
\usepackage{verbatim}
\usepackage{enumitem}
\usepackage{hyperref}
\usepackage{ifthen}
\usepackage{tikz}
\usepackage{psfrag}
\usepackage[titletoc]{appendix}
\usepackage{color}
\usepackage[top=4cm,bottom=4.5cm,left=4cm,right=4cm]{geometry}

\def\RR{{\rm I\hspace{-0.50ex}R} }

\newcommand{\ve}{\varepsilon}
\def\d{\partial}

\DeclareMathOperator{\dist}{dist}
\DeclareMathOperator{\diam}{diam}

\newcommand{\abs}[1]{\lvert#1\rvert}
\newcommand{\norm}[1]{\left\|#1\right\|}
\newcommand{\set}[1]{\left\{#1\right\}}
\newcommand{\inner}[2]{\left<#1,#2\right>}
\newcommand{\supp}{\mathop{\mathrm{supp}}}

\newcommand{\blue}[1]{{}}

\newtheorem{theorem}{Theorem}[section]
\newtheorem{lemma}[theorem]{Lemma}

\newtheorem{corollary}[theorem]{Corollary}
\newtheorem{remark}[theorem]{Remark}
\newtheorem{definition}[theorem]{Definition}
\newtheorem{notation}{Notation}
\newtheorem{assumption}[theorem]{Assumption}

\newenvironment{proof}{{\textbf{Proof.}}}{\hfill \textbf{$\square$}\vspace{0.2cm}}

\title{A Direct Method for Photoacoustic Tomography with Inhomogeneous Sound Speed}

\author{Zakaria Belhachmi
\thanks{Laboratoire de Math\'ematiques LMIA, Universit\'e de Haute Alsace, 4, rue des Fr\`eres Lumi\`ere,
68200 Mulhouse, France.(\tt{zakaria.belhachmi@uha.fr})} \and
Thomas Glatz
\thanks{Computational Science Center, University of Vienna, Oskar-Morgenstern Platz 1, A-1090 Vienna, Austria.
(\tt{thomas.glatzl@univie.ac.at})}\and
Otmar Scherzer
\thanks{Computational Science Center, University of Vienna, Oskar-Morgenstern Platz 1, A-1090 Vienna, Austria,
and Johann Radon Institute for Computational and Applied Mathematics (RICAM),
Austrian Academy of Sciences, Altenbergerstra\ss{}e 69,
A-4040 Linz, Austria.(\tt{otmar.scherzer@univie.ac.at})}
}

\begin{document}

\maketitle

\begin{abstract}
The standard approach for photoacoustic imaging with variable speed of sound is \emph{time reversal}, which 
consists in solving a well-posed final-boundary value problem \blue{for the wave equation} backwards in time.
This paper investigates the iterative Landweber regularization algorithm, where convergence is guaranteed by 
standard regularization theory, notably also in cases of \emph{trapping} sound speed or for short measurement times.
We formulate and solve the direct and inverse problem on \blue{the \emph{whole} Euclidean space}, what is 
common in standard photoacoustic imaging, but not for time-reversal algorithms, \blue{where the problems are considered 
on a domain enclosed by the measurement devices}.
We \blue{formulate} both the direct and adjoint photoacoustic operator \blue{as the solution of} an interior and an exterior 
\blue{differential} equation \blue{which are coupled by transmission conditions}.
The prior is solved \blue{numerically} using a Galerkin scheme in space and finite difference discretization in time, 
while the latter \blue{consist in solving a} boundary integral equation. We therefore use a BEM-FEM approach 
for numerical solution of the forward operators. We analyze this method, prove convergence, and provide numerical 
tests. Moreover, we compare the approach to time reversal.
\end{abstract}
\section*{Introduction}
Photoacoustic Imaging (PAI) is a novel imaging technique that allows for three dimensional imaging of 
small biological or medical specimens with high spatial resolution. 
It utilizes that an object expands and emits ultrasound waves when it is exposed to a short pulse
of electromagnetic radiation (see e.g. \cite{XuWan06,Wan09}). 
The emitted ultrasound \blue{pressure} is assumed to be proportional to the electromagnetic \emph{absorption density} 
which provides detailed anatomical and functional information. 
PAI aims for visualization of the absorption density by using measurements of the emitted wave outside of the object.

Opposed to the standard PAI problem \cite{KucKun08,Wan08,Wan09} we assume a spatially varying speed of sound. 
The underlying mathematical model of the wave propagation with spatially varying sound speed $c(x)$ is 
the \emph{acoustic wave equation}
\begin{align}\label{eq:wave_intro}
\begin{aligned}
\frac1{c^2(x)}y''(x,t) - \Delta y(x,t) & = 0 \text{ in } \RR^n  \times (0,\infty)\,,\\
        y(x,0) & = f(x)  \text{ in } \RR^n\,,\\
        y'(x,0) &= 0 \text{ in } \RR^n
\end{aligned}
\end{align}
where $f$ denotes the absorption density, which is proportional to the absorbed electromagnetic energy.

The \emph{inverse problem} of photoacoustics with variable wave speed consists in determining the function $f$ from 
measurement data $m$ of $y$ on a surface $\partial \Omega$ over time $(0,T)$. That is, it consists in solving the 
equation 
\begin{equation}\label{eq:meas_intro}
y \vert_\Sigma = m\;.
\end{equation}
With constant sound speed there exists a variety of analytical reconstruction formulas and numerical inversion techniques. 
We mention \cite{XuFenWan02,XuFenWan03,XuWan02a,XuWan02b,XuXuWan02}, see also the surveys \cite{KucKun08,Wan09,KucKun11,WanAna11}.
In the case of inhomogeneous sound speed, exact reconstruction formulas have been derived in \cite{AgrKuc07}.
A numerical method, providing approximations, is \emph{time reversal} \cite{Fin92,GruHalPalBur07,GruHofHalPalBur07,HriKucNgu08}, 
\blue{where a space-time boundary value problem for the wave equation is solved backwards in time}. 
Thereby the measurement data serve as Dirichlet boundary data, 
and the \emph{initial} conditions at final observation time $T$ are assumed to be identically zero. 
On top of the original algorithm of Fink \cite{Fin92}, the approach in \cite{SteUhl09} suggests a 
harmonic extension of the boundary data at time $T$ as initialization data.
The reconstruction obtained by time-reversal, lets say $f_T$, approximates the true solution $f$ as $T$ increases 
to infinity. The exact $f$ can be reconstructed if $y(\cdot,T) \equiv 0$, which happens, aside from 
trivial cases, only in odd dimensions and for homogeneous sound speed.
In all other cases the results are error-prone.
Moreover, this method produces approximations of $f$ only under \emph{non-trapping} conditions on $c$ (see \cite{HriKucNgu08,Hri09}).

Stefanov and Uhlmann \cite{SteUhl09} showed that if $\diam(\Omega) = 2T_0$ (where $\diam(\Omega)$ denotes the diameter of $\Omega$ 
with respect to the Riemannian metric $c^{-2}dx$), then an observation time $T > T_0$ is sufficient for a unique reconstruction of $f$. 
However, for stability of the inverse problem one needs longer measurement times and a non-trapping speed of sound condition. 
In fact, if the measurement time $T$ is larger than a certain threshold, depending on the longest geodesic in the metric induced by $c$, 
the algorithm presented in \cite{SteUhl09} provides a theoretically exact reconstruction in terms of a Neumann series 
that contains multiple, subsequent time reversal and forward propagation of the data term. 
A computational realization of this approach has been investigated in \cite{QiaSteUhlZha11}. 
This algorithm serves as a benchmark for our proposed algorithm.

The presented approach for PAI inversion with variable sound speed relies on linear regularization theory \cite{Gro84}. 
Specifically, we obtain regularized convergence to the \emph{minimum norm solution} even for short measurement times.
Moreover, we obtain a new reconstruction method that does not require an artificial cut-off of the measurement data, 
nor harmonic extension of the data at the final observation time $T$.
\blue{
The approach can be regarded as natural generalization of backprojection,
as it uses the adjoint operator in the backpropagation step of the reconstruction.
As in backprojection, the solution of the adjoint problem is computed on an unbounded spatial domain,
rather than solving a boundary value problem in time reversal.
}

For the numerical computations, we decouple the wave equation into an interior part (solved by finite element methods), 
and an exterior part (with homogeneous sound speed) that is rewritten in terms of a boundary integral formulation. 
We then solve the coupled BEM-FEM system numerically. 
Note that by this approach we use exact, non-reflecting boundary conditions \cite{AbbJolRodTer11,FalMon14} and therefore avoid the necessity of a \emph{perfectly matched layer} to deal with the cut-off outside the domain of interest.
\blue
{Since this formulation allows for inhomogeneous transmission conditions, 
it is also well-suited for the solution of the adjoint problem.}
The results are compared to conventional time reversal and the Neumann-series approach.

The paper is organized as follows: 
In Section \ref{sec:direct} we formulate the direct photoacoustic operator $L$ in a suitable choice of function spaces, 
and derive the adjoint.
In Section \ref{sec:Landweber} we give a short overview about Landweber iteration and review some regularization 
results regarding convergence and convergence rates in view of PAI reconstruction.
Moreover, we discuss the relation to the Neumann series approach in \cite{SteUhl09}.
In Section \ref{sec:fem_bem}, we formulate the transmission problem for the wave equation used for numerical computation of 
both the direct and the adjoint problem. We state the boundary relations used for taking into account the exterior domain. 
In addition, we briefly describe the used discretization.  
Finally, Section \ref{sec:results} provides a comparison of our reconstruction algorithm with the state-of-the-art reconstruction 
by time reversal and the \emph{enhanced} time reversal Neumann series method from \cite{SteUhl09}.

%
All along this paper we use the following notation and abbreviations:
\begin{notation}\label{not:spaces_trace}
Let $\Omega$ be a non-empty, open, bounded and connected domain in $\RR^n$ with $C^1$-boundary $\partial \Omega$.
The vector $\bm n(x)$, with $x\in \partial \Omega$, denotes the outward pointing unit normal vector.
We use the following sets
\begin{equation*}
 \Omega^+ := \RR^n \backslash \overline{\Omega},\;\Omega^- := \Omega \text{ and } \Sigma :=  \partial \Omega  \times (0,T)\;.
\end{equation*}
We use the following Hilbert spaces:
\begin{itemize}
\item $L^2(\Omega) = \set{\rho \in L^2(\RR^n) : \rho \equiv 0 \text{ in } \RR^n \backslash \Omega}$, with inner product
      \[\inner{\rho_1}{\rho_2}_{L^2(\Omega)}  = \int_{\RR^n} \rho_1(x) \rho_2(x)\,dx\;.\]
\item For $\hat{\Omega} = \Omega$ or $\RR^n$:
      \begin{itemize}
      \item Let $H_0^1(\hat{\Omega})$ be the closure of differentiable functions on $\RR^n$ with compact support in 
            $\hat{\Omega}$, associated with the inner product
            \[\inner{\rho_1}{\rho_2}_{H_0^1(\hat{\Omega})} = \int_{\RR^n} \nabla \rho_1(x)\cdot \nabla \rho_2(x)\,dx\;.\]
      \item $H^1(\hat{\Omega})$ denotes the standard Sobolev space with inner product
            \[\inner{\rho_1}{\rho_2}_{H^1(\hat{\Omega})} = \int_{\hat{\Omega}} \rho_1(x) \rho_2(x) + \nabla \rho_1(x)\cdot \nabla \rho_2(x)\,dx\;.\]
\end{itemize}
\item $L^2(\partial \Omega)$ denotes the standard Hilbert space of square integrable functions on $\partial \Omega$ with inner product 
      \begin{equation*}
            \inner{\phi_1}{\phi_2}_{L^2(\partial \Omega)} = \int_{\partial \Omega} \phi_1(x) \phi_2(x)\,dS(x)\;.
      \end{equation*}
      $L^2(\Sigma)$ denotes the standard Hilbert space of square integrable functions on $\Sigma$ with inner product 
      \begin{equation*}
            \inner{\phi_1}{\phi_2}_{L^2(\Sigma)} = \int_0^T \int_{\partial \Omega} \phi_1(x,t) \phi_2(x,t)\,dS(x) dt\;.
      \end{equation*}
\item The induced norms are denoted by $\norm{\cdot}_{L^2(\Omega)}$, $\norm{\cdot}_{L^2(\Sigma)}$, $\norm{\cdot}_{H^1(\hat{\Omega})}$ and 
      $\norm{\cdot}_{H_0^1(\hat{\Omega})}$, respectively.
      
      For $\Omega$ bounded, on $H_0^1(\Omega)$, the norms $\norm{\cdot}_{H_0^1(\Omega)}$ and $\norm{\cdot}_{H^1(\Omega)}$ are equivalent (see \cite[Theorem 6.28]{Ada75}):
      \begin{equation}
      \label{eq:equiv}
      C_0 \norm{\rho}_{H^1(\Omega)} \leq \norm{\rho}_{H_0^1(\Omega)} \leq \norm{\rho}_{H^1(\Omega)} \quad \text{ for all } \rho \in H_0^1(\Omega).
      \end{equation}
\item The trace operator 
      $\gamma_{\Omega} : H^1(\RR^n) \to L^2(\partial \Omega)$ restricts functions defined on $\RR^n$ 
      onto $\partial \Omega$, respectively. 
      This operator is the decomposition of the standard trace operator 
      \begin{equation*}
       \gamma: H^1(\Omega) \to L^2(\partial \Omega)
      \end{equation*}
      and the restriction operator 
      \begin{equation*}
       R: H^1(\RR^n) \to H^1(\Omega)\,,
      \end{equation*}
      and thus as a composition of two bounded operators \cite[Theorem 5.22]{Ada75} bounded. 
      We abbreviate the norm with 
      \begin{equation}
       \label{eq:norm}
       \mathcal{C}_\gamma:=\norm{\gamma \circ R}\;.
      \end{equation}
\end{itemize}
\end{notation}

\begin{notation}
The absorption density $f$ and the sound speed $c^2$ are supposed to satisfy: 
\begin{itemize}
 \item $c \in C^1(\overline{\Omega})$, satisfies $0 < c_{min} \leq c(x) \leq c_{max}$ and $c$ is non-constant in $\Omega$. 
 \item Without loss of generality we assume that $c\equiv 1$ in $\Omega^+$. 
 \item The absorption density function $f \in H_0^1(\Omega)$ is compactly supported in $\Omega$: \\
 $\text{supp}(f) \subseteq \Omega$.
\blue{ \item We denote by
 $|\Omega|:=\int_\Omega 1 dx$ the area of $\Omega$.} 
\end{itemize}
For the sake of simplicity of notation we omit space and time arguments of functions whenever this is convenient.
\end{notation}

\section{Direct Problem of Wave-Propagation}
\label{sec:direct}
We analyze the \emph{wave operator} $L$ mapping the absorption density $f$ onto the solution 
$y$ of the wave equation \eqref{eq:wave_intro} restricted to $\Sigma$. That is
\begin{equation}\label{obs_op}
L: H_0^1(\Omega)  \rightarrow L^2(\Sigma)\,, \quad  f \mapsto y|_\Sigma\;.
\end{equation}
In the following we show that $L$ is bounded. Let us write
\begin{align}\label{eq:wave_energy}
E(t) := \int_{\RR^n} \frac1{c^2}\,(y')^2 + \abs{ \nabla y}^2\,dx\;.
\end{align}
Computing the derivative of $E$ with respect to $t$ and taking into account \eqref{eq:wave_intro} gives 
\begin{equation*}
E'(t) = 2 \int_{\RR^n}\frac1{c^2}\, y'' y' - \Delta y\, y'\,dx = 0\;.
\end{equation*}
Consequently
\begin{equation}
\label{eq:E}
E(t) = E(0) = \norm{f}_{H_0^1(\Omega)}^2\,,
\end{equation}
which implies that 
\begin{align}\label{eq:wave_bound1}
\frac1{\blue{c^2_\text{max}}}\int_{\RR^n} (y')^2\,dx \leq \int_{\RR^n} \frac1{c^2}\,(y')^2\,dx \leq\blue{ E(t) =} \norm{f}_{H_0^1(\Omega)}^2
\intertext{and}\label{eq:wave_bound2}
\norm{y(t)}_{H_0^1(\RR^n)}  \blue{= \int_{\RR^n}|\nabla y|^2 dx\leq E(t)}\leq \norm{f}_{H_0^1(\Omega)}
\end{align}
for every $t\in(0,T)$.
\begin{lemma}
 Let $y$ be the solution of \eqref{eq:wave_intro}, then 
 \begin{equation}
 \label{eq:fund_const}
  \norm{y(t)}_{H^1(\RR^n)} \leq \mathcal{C}(T)\norm{f}_{H_0^1(\Omega)}\,,\qquad 
  \text{ for all } t \in (0,T)\;.
 \end{equation}
 with 
 \begin{equation*}
  \mathcal{C}(T):= \sqrt{\blue{1+\frac2{C_0^2}\max\set{1,\blue{c_\mathrm{max}^2} T^2}}}\,,
 \end{equation*}
 where $C_0$ is defined in \eqref{eq:equiv}.
\end{lemma}
\begin{proof}
First, we note that for arbitrary $t \in (0,T)$, it follows from \eqref{eq:wave_bound1} that:
\begin{equation*}
 \begin{aligned}
  \int_{\RR^n} (y(x,t)-y(x,0))^2\,dx &= \int_{\RR^n} \left( \int_0^t y'(x,\hat{t})\,d\hat{t} \right)^2\,dx\\
  &\leq t \int_{\RR^n} \int_0^t (y'(x,\hat{t}))^2 \,d\hat{t} \,dx\\  
  &= t \int_0^t \int_{\RR^n} (y'(x,\hat{t}))^2 \,dx\,d\hat{t} \\
  &\leq \blue{c_\text{max}^2} t^2 \norm{f}_{H_0^1(\Omega)}^2 \\
  &\leq \blue{c_\text{max}^2} T^2 \norm{f}_{H_0^1(\Omega)}^2\;.
 \end{aligned}
\end{equation*}
Because $(a-b)^2 \geq \frac12 a^2 - b^2$ it follows from \eqref{eq:equiv} that 
\begin{equation*}
 \begin{aligned}
  \int_{\RR^n} (y(x,t))^2\,dx &\leq 2 \int_{\RR^n} (y(x,t)-y(x,0))^2\,dx + 2\int_{\RR^n} (y(x,0))^2\,dx\\
  &\leq 2 \blue{c_\text{max}^2} T^2 \norm{f}_{H_0^1(\Omega)}^2 + 2 \norm{f}_{L^2(\Omega)}^2\\
  &\leq 2 \max\set{1,\blue{c_\text{max}^2} T^2}\norm{f}_{H^1(\Omega)}^2\;.
 \end{aligned}
\end{equation*}
Because $f \in H_0^1(\Omega)$ it follows that 
\begin{equation*}
  \norm{y(t)}_{L^2(\RR^n)}^2 \leq \frac{2}{C_0\blue{^2}} \max\set{1,\blue{c_\text{max}^2}T^2}\norm{f}_{H_0^1(\Omega)}^2\;.
\end{equation*}
This, together with \eqref{eq:wave_bound2} shows that for all $t \in (0,T)$:
\begin{equation*}
\norm{y(t)}_{H^1(\RR^n)} \leq \sqrt{\blue{1+\frac2{C_0^2}\max\set{1,\blue{c_\text{max}^2} T^2}}}\norm{f}_{H_0^1(\Omega)}\;.
\end{equation*}
\end{proof}

In the following we prove boundedness of $L$:
\begin{theorem}
The operator $L : H_0^1(\Omega) \to L^2(\Sigma)$ is bounded and 
\begin{equation} \label{eq:norm_l}
\norm{L} \leq \mathcal{C}_\gamma \mathcal{C}(T)\sqrt{T}\;.
\end{equation}
\end{theorem}
\begin{proof}
For given $f$ let $y$ be the solution of \eqref{eq:wave_intro}). 
From \eqref{eq:wave_energy} it follows that the solution $y$ of \eqref{eq:fund_const} is in $H^1(\RR^n)$ for every $t > 0$.
Thus from \eqref{eq:norm} and \eqref{eq:fund_const} it follows that 
\begin{equation*}
\norm{y}_{\Sigma}^2 = \int_0^T \int_{\partial \Omega} y^2(t)\,d\sigma d t
\leq \mathcal{C}_\gamma^2 \int_0^T \norm{y(t)}_{H^1(\RR^n)}^2 d t \leq \mathcal{C}_\gamma^2 \mathcal{C}(T)^2 T \norm{f}_{H_0^1(\Omega)}^2\,,
\end{equation*}
which gives the assertion.

\begin{remark}[Injectivity of $L$]\label{rem:inj}
In order to obtain injectivity of $L$, we need $T$ to be sufficiently large. To specify this, we define
\begin{align}\label{eq:T0}
T_0:=\max\limits_{x\in\Omega}\left(\dist(x,\d\Omega)\right)\,,
\end{align}
where $\dist(x,\d\Omega)$ is the distance of $x$ to the closest point $x'\in\partial \Omega$
with respect to the Riemannian metric $c^{-2}dx$ (see also \cite{QiaSteUhlZha11}). 
From \cite[Thm. 2]{SteUhl09} it follows that if $T>T_0$, than $L[f]=0$ implies $f=0$ in $(0,T)\times\mathbb R^n$. 
\end{remark}
\end{proof}

In the following we characterize the adjoint of $L: H_0^1(\Omega) \to L^2(\Sigma)$ on a dense subset of $L^2(\Sigma)$.
Because we know from elementary functional analysis that $L^*: L^2(\Sigma) \to H_0^1(\Omega)$ is bounded, \blue{we get a 
characterization \blue{of the whole space} by limits of convergent sequences.}
\begin{definition}
Let $i$ be the embedding operator from $H_0^1(\Omega)$ to $L^2(\Omega)$. 
Then $i^*: L^2(\Omega) \to H_0^1(\Omega)$ is the operator 
which maps a function $\psi \in L^2(\Omega)$ onto the solution of the equation
\begin{equation*}
{-\Delta u} = \psi\text{ in } \Omega\,, \qquad u=0 \text{ on } \partial \Omega\;.
\end{equation*}
That is 
\begin{equation}\label{eq:Dirichlet}
 i^* = {- \Delta^{-1}}\,,
\end{equation}
where $\Delta$ is the Laplace-operator with homogeneous Dirichlet boundary conditions.
\end{definition}

In the following we derive the adjoint $L^*: $ of the operator $L$, which is required for 
the implementation of the Landweber iteration below.
\begin{theorem}
For $h \in C^\infty((0,T) \times \partial \Omega)$ the adjoint of the operator $L$, defined in \eqref{obs_op}, is given by
\begin{equation}\label{eq:adjoint0}
L^*[h] = i^* \circ L_D^*[h]
\end{equation}
where  
\begin{equation}
\label{eq:adjoint}
L_D^*[h] = \left. \frac1{c^2}\,z'(0)\right|_\Omega\,,
\end{equation}
and $z:=z(h)$ is the weak solution of
 \begin{equation}\label{eq:wave_adj}
 \begin{aligned}
         \frac1{c^2}\, z'' - \Delta z &=0 \text{ in } \RR^n \backslash \partial \Omega \times (0,T)\,,\\
          z(T) = z'(T) &=0  \text{ in } \RR^n, \\
          \left[ z \right] =0\,, \quad
        \left[  \frac{\partial z}{\partial \bm n} \right] &= h \text{ on } \partial \Omega \times (0,T)\;.
\end{aligned}
\end{equation}
Here 
\[ [z]:=z^+|_\Sigma-z^-|_\Sigma \text{ and }
   \left[   \frac{\partial z}{\partial \bm n}\right] := 
   \left.\frac{\partial z^+}{\partial \bm n}\right|_\Sigma - \left.\frac{\partial z^-}{\partial \bm n} \right|_{\Sigma}
   \]
where $z^+:=z|_{\Omega^+\times(0,T)}$ and $z^-:=z|_{\Omega\times(0,T)}$.
\end{theorem}
\begin{proof}
For $h \in C^\infty((0,T) \times \partial \Omega)$ the existence of a weak solution of \eqref{eq:weak_adjoint} is proven in the Appendix.
Taking $v=y$ where $y$ denotes the solution of \eqref{eq:wave_intro} it follows that
\begin{equation}
\label{eq:adjoint_proof}
\begin{aligned} 
\int_\Sigma h L[f]\,dS(x) dt &= \int_\Sigma h y\,dS(x) dt \\
&= \int_\Omega \frac{z'(0)}{c^2} f\,dx\\
 &= \int_\Omega \Delta \left[\Delta^{-1} \left[\frac{z'(0)}{c^2}\right]\right] f\,dx\\
 &= - \int_\Omega \nabla \left(\Delta^{-1} \left[\frac{z'(0)}{c^2}\right]\right) \cdot \nabla  f\,dx\\
 &= \int_\Omega \nabla i^* \left[\frac{z'(0)}{c^2}\right] \cdot \nabla  f\,dx\\
 &= \inner{i^* \left[\frac{z'(0)}{c^2}\right]}{f}_{H_0^1(\Omega)}
 \;.
\end{aligned}
\end{equation}
\end{proof}

\blue{Existence of a weak solution of the equation \eqref{eq:wave_adj} follows from the result in the Appendix, which 
is proven along the lines of \cite{Eva98}. This kind of weak solution, used here, requires in fact differentiable transmission data $h$.
There might exist weaker solution concepts, which directly guarantee existence of a solution if $h \in L^2$, but the currently used 
result is not applicable in this sense, and therefore we have to use a density argument to give a meaning to the solution of \eqref{eq:wave_adj} in the case that the transmission data is only in $L^2$.}

%
\section{Landweber Iteration for Solving the Inverse 
	Problem of Photoacoustics}\label{sec:Landweber}
The \emph{photoacoustic imaging problem} rewrites as the solution of the operator equation:
\begin{equation}
\label{eq:operator}
 L[f] = m\;.
\end{equation}
If the null-space of $L$ is non-trivial, then iterative regularization algorithms, in general, when $m$ is an element of the range of $L$ 
reconstruct the \emph{minimum norm solution}
\begin{equation}
\label{eq:mp}
 f^\dagger = L^\dagger[m]\,,
\end{equation}
where $L^\dagger$ denotes the Moore-Penrose inverse of $L$ (see \cite{Nas76} for a survey on Moore-Penrose inverses).

We propose to use the Landweber's iteration for solving \eqref{eq:operator}, because it can be compared with 
time reversal methods, which are the standard references in this field. More efficient regularization 
algorithms are at hand \cite{Han95}, but these are less intuitive to be compared with time reversal.
\blue{For photoacoustic reconstruction, a conjugate gradient approach 
incorporating a quadratic approximation of the sound speed was proposed in \cite{ModAnaRiv10}.}
 
\blue{In numerical applications and with constant speed of sound, 
variational methods, like TV (total variation) minimization, have been implemented \cite{ZhaWanZha12,DonGoeKun14}. 
Such an approach can lead to sharper reconstructions of $f$ when it is piecewise constant.   
However, so far, a profound theoretical analysis exists only for the spherical mean operator \cite[Proposition 3.82]{SchGraGroHalLen09} 
and is yet missing for the photoacoustic operator \eqref{obs_op}.}

In the following we review properties of the Landweber iteration in an abstract setting (see \cite{Gro84,EngHanNeu96}). 
We use the same notation for the abstract operator and the photoacoustic operator and measurement data $m$, $m^\delta$, respectively,
in order to have direct connection.

\subsection{Abstract Landweber Regularization}
\label{subsec:regularization}
Everything that is formulated below is based on the following assumption:
\begin{assumption}
Let $L: H_1 \to H_2$ be an operator between Hilbert spaces $H_1$ and $H_2$ satisfying $\omega \norm{L}^2 \leq 1$ for some 
$\omega >0$. 
Moreover, we assume that data $m^\delta$ of $m$ is available (typically considered as noisy data), which satisfy
\begin{equation}
 \label{eq:noise}
\norm{m-m^\delta}_{H_2} \leq \delta\;.
\end{equation}
\end{assumption}

Then the Landweber iteration reads as follows:
\begin{align}\label{eq:def_landweber}
\begin{aligned}
f_0:=0 \quad \text{and} \quad f_k^\delta=f_{k-1}^\delta-\omega L^*[L[f_{k-1}^\delta]-m^\delta]\,,\quad k=1,2,\dots\;.
\end{aligned}
\end{align}
In case $\delta=0$, that is, if $m^\delta = m$, then we write $f_k$ instead of $f_k^\delta$. 

Let $\tau > 1$ be some fixed constant. The Landweber iteration is only iterated for $k=1,2,\ldots$ as long as  
\begin{equation}
 \label{eq:stop}
 \norm{m^\delta - L[f_k^\delta]}_{H_2} > \tau\delta\;.
\end{equation}
The index, where \eqref{eq:stop} is violated for the first time is denoted by $k_*(\delta,m^\delta)$.

The following theorem shows that the Landweber iteration converges to the best-approximating solution:
\begin{theorem}\label{thm:Landweber1}
Let $m \in \mathcal{R}(L)$ (note that the range of $L$ equals the domain of $L^\dagger$).
\begin{itemize}
 \item Let $\delta = 0$, then the Landweber iterates $(f_k)$ (cf. \eqref{eq:def_landweber}) converge to 
       the $f^\dagger$, i.e., $\norm{f_k - f^\dagger}_{H^1} \to 0$.
       In addition, we have the series expansion:
       \begin{equation*}
       f^\dagger = \sum_{j=0}^\infty (I-\omega L^*L)^j [L^*[m]]\;.
       \end{equation*}
 \item For $\delta > 0$ and $m^\delta$ satisfying $\norm{m-m^\delta}_{H_2} \leq \delta$ let $k_*^\delta = k_*(\delta,m^\delta)\blue{-1}$ 
       as in \eqref{eq:stop}.
       Then 
       \begin{equation*}
        f_{k_*^\delta}^\delta \to_{H_1} f^\dagger\;.
       \end{equation*}
\end{itemize}
Moreover, if $m\notin \mathcal D(L^\dagger)$, then $\norm{f_k}_{H_1} \to \infty$ as $k \to \infty$. 
\end{theorem}
In the following we prove properties of the wave-operator $L$, such that we can apply the general regularization 
results.

\subsection{Convergence of the Landweber Iteration for the Photoacoustic Problem}

In the following we apply Theorem \ref{thm:Landweber1},  for the photoacoustic imaging problem.
\begin{corollary}
\label{cor:Landweber}
Let $\omega \leq \frac{1}{\mathcal{C}_\gamma^2 \mathcal{C}^2(T) T} = \mathcal{O}(1/T^3)$ and $L: H_0^1(\Omega) \to L^2(\Sigma)$ as in \eqref{obs_op}. 
Moreover, assume that 
\begin{equation*}
 \norm{m - m^\delta}_{L^2(\Sigma)} \leq \delta\;.
\end{equation*}
Then the Landweber iterates satisfy:
\begin{itemize}
 \item If $\delta =0$, then 
       \begin{equation*}
       f^\dagger = \sum_{k=0}^\infty (I-L^*L)^j [L^*[m]]\;.
       \end{equation*}
 \item For $\delta > 0$, the Landweber iteration is terminated at $k_*(\delta) := k_*(\delta,m^\delta)\blue{-1}$ according to \eqref{eq:stop}.
       Then 
       \begin{equation*}
        f_{k_*^\delta}^\delta \to f^\dagger \text{ for } \delta \to 0\;.
       \end{equation*}
 \item If $T>T_0$, the reconstruction is unique, and $f_{k_*(\delta)}^\delta$ converges to the unique solution.
\end{itemize}
\end{corollary}
\begin{proof}
First, we note that from \eqref{eq:norm_l} it follows that $\omega \norm{L}^2 \leq 1$. 
Then, the first two items follow directly from Theorems \ref{thm:Landweber1}. 

From the injectivity of $L$ for $T>T_0$ (see Remark \ref{rem:inj}) it follows that $f^\dagger=f$, which implies 
unique reconstruction.
\end{proof}

\subsection{Comparison with Time Reversal}
We compare our approach with different variants of \emph{time reversal}. 
We formally define the time reversal operator:
\begin{align}\label{eq:TR_operator}
\bar{L}[h] = z(\cdot,0)\,,
\end{align} 
where $z$ is a solution of 
\begin{align}\label{eq:TR}
 \begin{aligned}
         \frac1{c^2} z''- \Delta z &=0 \text{ in } \Omega \times (0,T),\\
         z(T) = z'(T) &=0  \text{ in } \Omega , \\         
         z  &= h \text{ on } \partial \Omega \times (0,T),\\
\end{aligned}
\end{align}
The fundamental differences between $\bar{L}$ and $L^*$ are that $\bar{L}$ is defined for functions with support in $\Omega$ 
and that therefore $L^*$ requires a transmission condition in its definition.

Stefanov and Uhlmann \cite{SteUhl09} modified the time reversal approach in the following sense: 
Rather than assuming (in most cases unjustified) the initial data $z(T) \equiv 0$,
they used the harmonic extension of the data term $h(s,T)$, for $s \in \partial \Omega$, as initial datum at $T$.
That is, for 
\begin{equation*}
 - \Delta \phi = 0 \text{ in } \Omega\,, \text{ with } \phi(\cdot) = m(\cdot,T) \text{ on } \partial \Omega
\end{equation*}
the modified time-reversal operator 
\begin{align}\label{eq:TR_operator_Uhl}
\tilde{\bar L}[h] = z(\cdot,0)
\end{align} 
is defined by the solution of equation 
\begin{align}\label{eq:TR_Uhl}
 \begin{aligned}
         \frac1{c^2} z''- \Delta z &=0 \text{ in } \Omega \times (0,T),\\
          z(T) = \phi\,,\quad z'(T) &=0 \text{ in } \Omega , \\
        z  &= h \text{ on } \partial \Omega\times(0,T)\;.\\
\end{aligned}
\end{align}
They were able to show that under non-trapping conditions and for sufficiently large measurement time $T$, there exists a 
compact operator $K:H_0^1(\Omega)\to H_0^1(\Omega)$ satisfying $\norm{K}<1$, and
\begin{align}\label{eq:Neumann1}
\tilde{\bar{L}} L  = Id - K\,,
\end{align}
Therefore, the initial condition $f$ can be expanded into the Neumann series
\begin{align}\label{eq:Neumann2}
f=\sum_{j=0}^\infty K^j[m]\;.
\end{align}
By induction, it is easy to see that the $m$-th iterate can be written as
\begin{align}\label{eq:Neumann3}
f_k=f_{k-1}-\tilde{\bar{L}}[L[f_{k-1}] - m]\,,
\end{align}
where
\[
f_k = \sum_{j=0}^{\blue{k}} K^j [m]\;.
\]
At this point, we emphasize on the structural similarity between \eqref{eq:Neumann3} and \eqref{eq:def_landweber}.

\begin{remark}
 We emphasize that for time-reversal there is no theory on stopping in case of error-prone data, such as we have available for the Landweber iteration.
 \blue{In fact, the boundary data have to be in $H^1(\Sigma)$, which follows from $f\in H^1_0(\Omega)$ under the additional assumption that $\supp f\subset\Omega$ (see \cite[Remark 5]{SteUhl09}.}
\end{remark}
\blue{
\begin{remark}
Compared to time reversal, the smoothness assumptions to the data term can be relaxed when searching for a regularized solution. 
Indeed, instead of needing $H^1$-data in the range of $L$ as in time reversal, the present approach gives regularized convergence for a data-term in $L^2$.
The smoothness assumption on the coefficient $c$ is mainly of technical nature, to be able to apply the standard results of microlocal analysis.
A careful but tedious analysis might give similar results with less smoothness requirements. 
The stability for piecewise smooth coefficients has been analyzed in \cite{SteUhl11}.
\end{remark}
}

\section{Numerical realization of $L$ and $L^*$}\label{sec:fem_bem}
We solve \eqref{eq:wave_intro} and \eqref{eq:wave_adj} with the same numerical framework. 
By changing the variable $t\to T-t$ in \eqref{eq:wave_adj}, both equations can be rewritten as the 
transmission problem:
\begin{align}\label{eq:transmission}
\begin{aligned}
\frac{1}{c^2}v'' - \Delta v &= 0\qquad \text{in }\RR^n \backslash \partial \Omega\times(0,T)\,,\\
v(0)=v_0\,,\quad v'(0) &= 0\qquad \text{in }\RR^n\backslash \partial \Omega\,,\\
\left[\frac{\partial v}{\partial \bm n}\right] = \rho\,,\quad [v] &= 0
\qquad \text{on }\Sigma\;.
\end{aligned}
\end{align}
where for \eqref{eq:wave_intro} $\rho \equiv 0$ and $v_0 = f \in H_0^1(\Omega)$ and for \eqref{eq:wave_adj} we have 
$\rho \equiv h \in L^2(\Sigma)$ and $v_0 \equiv 0$\;. 

Let $v_0^- = v_0$ in $\Omega^-$, $v_0^+ = 0$ in $\Omega^+$, then $v^\pm = v|_{\Omega^\pm}$ satisfy, respectively:
\begin{equation}\label{eq:interior}
\begin{aligned}
\frac{1}{c^2} v_\pm'' - \Delta v_\pm &= 0\qquad \text{in } \Omega^\pm\times(0,T)\,,\\
v_\pm(0) =v_0^\pm\,,\quad  v_\pm'(0) &= 0\qquad \text{in } \Omega^\pm\,,\\
\end{aligned}
\end{equation}
together with the transmission conditions
\begin{equation}
\left[\frac{\partial v}{\partial \bm n}\right] = \rho \text{ and }
[v] = 0 \qquad \text{on }\Sigma\;.
\end{equation}

Let $G$ denote the fundamental solution of the standard wave equation with $c^2 \equiv 1$ in $\RR^n$. 
It is defined in $\RR^n\times \RR$, and its explicit expression
\begin{align*}
G(x,t) = 
\left\{\begin{array}{cl}
\frac{H(t-|x|)}{2\pi\sqrt{t^2-|x|^2}},\quad& n=2\,,\\
\frac{\delta(t-|x|)}{4\pi|x|},\quad& n=3\,,
\end{array}
\right.
\end{align*}
with $H$ denoting the Heaviside step function, and $\delta(\cdot)$ being the 3D Dirac delta distribution (see, e.g., \cite{Fri75}).

The \emph{(retarded)} \emph{single-} and \emph{double- layer potentials} for $(x,t)\in\Sigma$ are defined by
\begin{align}\label{eq:RP_operators}
\begin{aligned}
\mathcal V [\varphi](x,t)&:=\int_0^t\int_{\d\Omega} G(y-x,t-\tau)\varphi(y,\tau)dS_yd\tau\,,\\
\mathcal{K}[\psi](x,t)&:=\int_0^t\int_{\d\Omega} \frac{\d G(y-x,t-\tau)}{\d \bm n}\psi(y,\tau)dS_yd\tau\;.
\end{aligned}
\end{align}
Because we assume that $c=1$ outside of $\Omega$ and because we assume that $\partial \Omega$ is a $C^{1}$ boundary, 
it follows that $v^+$ satisfies (see \cite{Duo03}):
\begin{equation}\label{eq:boundary_relation}
\frac12 v^+(x,t) = -\mathcal V \left[\frac{\d v^+}{\d \bm n}\right](x,t) - \mathcal{K}\left[ v^+ \right](x,t)\,,\quad\text{ for all } (x,t)\in\Sigma\;.
\end{equation}
The transmission conditions imply that 
\begin{equation*}
v^+=v^- \text{ and } \frac{\partial v^+}{\d \bm n} = \frac{\partial v^-}{\d \bm n} + \rho \text{ on } \Sigma\;.
\end{equation*}
Therefore the equation for $v=v^-$ on $\Omega=\Omega^-$ can be rewritten as follows:
\begin{align}\label{eq:cont_system}
\begin{aligned}
\frac{1}{c^2} v'' - \Delta v &= 0\qquad \text{in }(\Omega)\times(0,T)\,,\\
v(0) =v_0\,,\quad v'(0) &= 0\qquad \text{in }\Omega\,,\\
\frac12 v + \mathcal V \left[\frac{\partial v}{\d \bm n} + \rho\right] \,+\mathcal{K}[v]&= 0 \qquad \text{on }\Sigma\;.\\
\end{aligned}
\end{align}
The numerical solution is based on a weak formulation of \eqref{eq:cont_system}.
Integrating over $\Omega$, multiplying by $w\in H^1(\Omega)$ and integration by parts of the first line of \eqref{eq:cont_system} gives
\[
\frac{d^2}{dt^2} \inner{\frac{1}{c^2}v(t)}{w}_{L^2(\Omega)} + \inner{\nabla v(t)}{\nabla w}_{L^2(\Omega)}-\inner{\frac{\d v}{\d\bm n}(t)}{w}_{L^2(\d\Omega)} = 0\;.
\]
Additionally we introduce the unknown function $\lambda:=\frac{\d v}{\d\bm n}$ defined on $\Sigma$.
We are therefore searching for a solution $(v,\lambda)\in H^1(\Omega)\times H^{-1/2}(\partial \Omega) $ that satisfies for almost all $t\in (0,T)$ the system
\begin{align}\label{eq:cont_system_weak}
\begin{aligned}
\frac{d^2}{dt^2} \inner{\frac{1}{c^2}v(t)}{w}_{L^2(\Omega)} + \inner{\nabla v(t)}{\nabla w}_{L^2(\Omega)}-\inner{\lambda}{w}_{L^2(\d\Omega)} &= 0\quad\quad\text{ for all } w\in H^1(\Omega)\,,\\
v(0) =v_0\,,v'(0) &= 0\qquad \text{in }\Omega\,,\\
\frac12 v(t) + \mathcal V [\lambda+\rho](t)\,+\mathcal{K}[v](t)&= 0 \qquad \text{on }\Sigma\;.
\end{aligned}
\end{align}
For detailed analysis of the discretization of equations of that type \eqref{eq:cont_system_weak} we refer 
to \cite{FalMon14}, since we closely follow their approach.

Next, we discuss the time discretization of the integrals 
$\mathcal V [\lambda]$ and $\mathcal{K}[v]$. 
We consider a uniform time discretization of the interval $(0,T)$ into $N$ steps of length $\Delta_t = T/N$, 
and defining the discrete time levels $t_n = n\Delta_t$. 
Following \cite{FalMonScu12}, by using Lubich's convolution quadrature formula \cite{Lub88}, 
we approximate $\mathcal V [\lambda]$ and $\mathcal{K}[v]$ at the time steps $t_n,\;n=0,\dots,N$ by
\begin{align}\label{eq:RP_semidiscrete}
\begin{aligned}
\mathcal V [\varphi](x,t_n)\approx\sum_{j=0}^n\int\limits_{\d\Omega} w_{n-j}^\mathcal{V}(\Delta_t,|x-y|) \varphi(y,\blue{t_j})dS_y\,,\\
\mathcal{K}[\psi](x,t_n)\approx\sum_{j=0}^n\int\limits_{\d\Omega} w_{n-j}^\mathcal{K}(\Delta_t,|x-y|)\psi(y,\blue{t_j})dS_y\;.
\end{aligned}
\end{align}
where the coefficients $w_n^\mathcal{V}$ and $w_n^\mathcal{K}$ satisfy
\begin{align}\label{eq:Lubich_coeff}
\begin{aligned}
w_n^\mathcal{V}(\Delta_t,|x-y|)&=\frac{\beta}{2\pi L}\sum_{l=0}^{L-1}K^\mathcal{V}\left(|x-y|,\frac{\gamma(\rho\exp(\mathrm i l2\pi/L))}{\Delta_t}\right)\exp(-\mathrm i n2\pi/L)\,,\\
w_n^\mathcal{K}(\Delta_t,|x-y|)&=\frac{\beta}{2\pi L}\sum_{l=0}^{L-1}K^\mathcal{K}\left(|x-y|,\frac{\gamma(\rho\exp(\mathrm i l2\pi/L))}{\Delta_t}\right)\exp(-\mathrm i n2\pi/L)\,,
\end{aligned}
\end{align}
where
\begin{equation*}
K^{\mathcal V}(r,s) = K_0(rs)\,,\qquad K^\mathcal{K}(r,s) =- sK_1(rs)\frac{\partial r}{\partial \bm n}\,,
\end{equation*}
and $K_0(\cdot),K_1(\cdot)$ are the second kind modified Bessel functions of \blue {order $0$ and order $1$}, respectively.
The function $\gamma$ is given by $\gamma(z)=3/2-2z+1/2 z^2$ and is the associated characteristic quotient 
of the used backward differentiation formula method of order two. 
For the involved constants we choose $L=2N$ and $\beta=\epsilon^{1/2N}$, 
where $\epsilon$ is the machine precision (see \cite{Lub88,FalMonScu12} for more details).

In the first equation in \eqref{eq:cont_system_weak}, the second time derivative is approximated using the second order central difference expression
\[
\frac{d^2}{dt^2}v^n = \frac1{\Delta t^2}\left( v^{n+1}-2v^n+v^{n-1} \right ) + \mathcal{O}(\Delta t^2)\,,
\] 
where $v^n:=v(.,t_n)$. The first time derivative occurring in the initial condition is also discretized by central differences, namely
\[
\frac{d}{dt}v^0 = \frac1{\blue{2}\Delta t}\left( v^{1}-v^{-1} \right ) + \mathcal{O}(\Delta t^2)\;.
\] 
 From that we can restate the initial conditions as
\begin{align}\label{eq:init_cond_time_discr}
\begin{aligned}
\inner{v^0}{w}&=\inner{v^0}{w}\,,\\ 
\inner{v^1}{w}&=\inner{\frac1{c^2}v^0}{w} - \frac12 \Delta t^2\left(\inner{\nabla v^0}{\nabla w} \blue{-} \inner{\lambda^0}{w}\right), \quad\text{ for all } w\in H^1(\Omega)\;.
\end{aligned}
\end{align} 
Using this, the explicit Euler discretization of \eqref{eq:cont_system_weak} is stated as
\begin{align}\label{eq:semidiscrete}
\begin{aligned}
~  \frac1{\Delta t^2}\inner{\frac{1}{c^2}(v^{n+1}-2v^n+v^{n-1})}{w} + \inner{\nabla v^n}{\nabla w} - \inner{\lambda^n}{w} &= 0,
 \quad\text{ for all } w\in H^1(\Omega)\,,\\
~  \frac12 v^{n+1}(t) + \blue{\mathcal{V}[\lambda+\rho](t_{n+1})\,+\mathcal{K}[v](t_{n+1})} &= 0 \qquad \text{on }\Sigma\,,\\
\end{aligned}
\end{align}
for all $1 \leq n \leq N$.
\blue{Note that in \eqref{eq:init_cond_time_discr} and \eqref{eq:semidiscrete} we write $\inner{\cdot}{\cdot}$ shorthand for the $L^2$ inner product over either $\Omega$ or $\partial \Omega$.}

In space $\Omega$ is triangulated, and we use piecewise quadratic basis functions, supplemented by one cubic function for $v$. The functions $\Lambda$ and $v|_{\d\Omega}$ are discretized by the use of piecewise linear functions.
With this ansatz a mass-lumped integration scheme can be robustly implemented, 
which is not the case for purely piecewise quadratic functions. 
The details and numerical analysis to this scheme can be found in \cite{CohJolRobTor01}.

\blue{
\begin{remark}[Advantages and limitations]
The proposed numerical technique has the potential to be used for very general types of measurement surfaces 
and does not need an artificial cut-off at the boundary  \cite{FalMon14}. 
However, it relies on the assumption that the initial source is in $H^1$. 
The method naturally introduces errors in presence of large gradients, as occurring at the edges of piecewise constant images. 
In such cases, a very fine mesh has to be chosen to avoid oscillations in the solution, at the cost of performance. 
In these cases, the use of non-conforming methods as discontinuous Galerkin might improve the result \cite{Coc03}.
\end{remark}}

\section{Numerical experiments and results}
\label{sec:results}

Here, we present numerical results for the Landweber reconstructions and compare it with 
standard time-reversal reconstructions and the Neumann-series-approach \eqref{eq:Neumann2} - 
always computed with the same discretization. 
We concentrate in particular on the cases where classical time reversal techniques have their 
theoretical and practical drawbacks, that is, so-called \emph{trapping speed geometries} and short measurement times. 
In these cases, neither time reversal nor the Neumann-series \eqref{eq:Neumann2} provide 
theoretical convergence. As we have shown in Corollary \ref{cor:Landweber}, the Landweber reconstruction 
converges to the least squares solution, regardless of the chosen measurement time. 
In what follows, we also show the practical applicability of our method in such cases.

To create the data sets, we use the discretization of the coupled method introduced in Section \ref{sec:fem_bem}. 
The forward computations are performed in a circle of radius 1. \blue{The simulation mesh size is $h=0.013$ for the  {ghosts} phantom and $h=0.009$ for the Shepp-Logan phantom. 
In the reconstructions we use $h_r=0.025$ for the ghosts and $h_r=0.0095$ for the Shepp-Logan.}
For both simulation and reconstructions, the chosen finite element space is that of continuous, piecewise quadratic functions in space. 
For the spatial discretization of the boundary terms, we take piecewise linear basis functions.
Note that for convenience and to optimize the computational effort, we assume $\partial \Omega$ to be a circle of Radius $R$.
For the time discretization, wo choose the step size $\Delta t=h/(15c_\text{max})$ for the simulation and $\Delta t_r=h_r/(14c_\text{max})$. 
Here $c_\text{max}$ is the maximum speed of sound.
The time reversal reconstructions are obtained by solving the initial boundary value problem \eqref{eq:TR} with homogeneous initial 
values. In the images, when we used the harmonic extension time-reversal \eqref{eq:TR_Uhl}, the heading is \emph{Neumann}. 
The Landweber reconstruction is performed by the scheme described in Corollary \ref{cor:Landweber}.
The choices of $c$ in both the trapping and non-trapping case have been taken 
from \cite{QiaSteUhlZha11}.

\subsection{Non-trapping sound speed}
\begin{figure}[!ht]\center
\includegraphics[width=4.5cm]{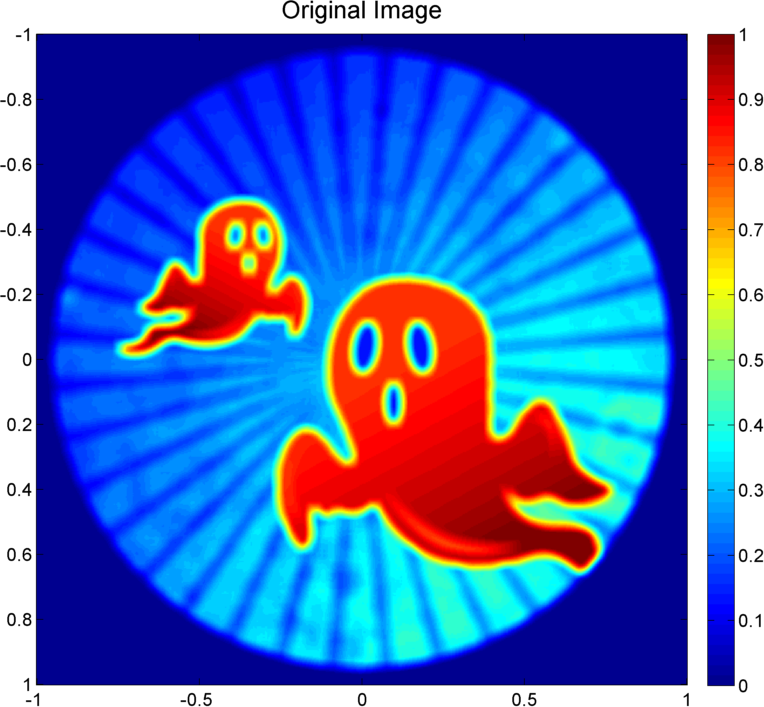}
\includegraphics[width=4.5cm]{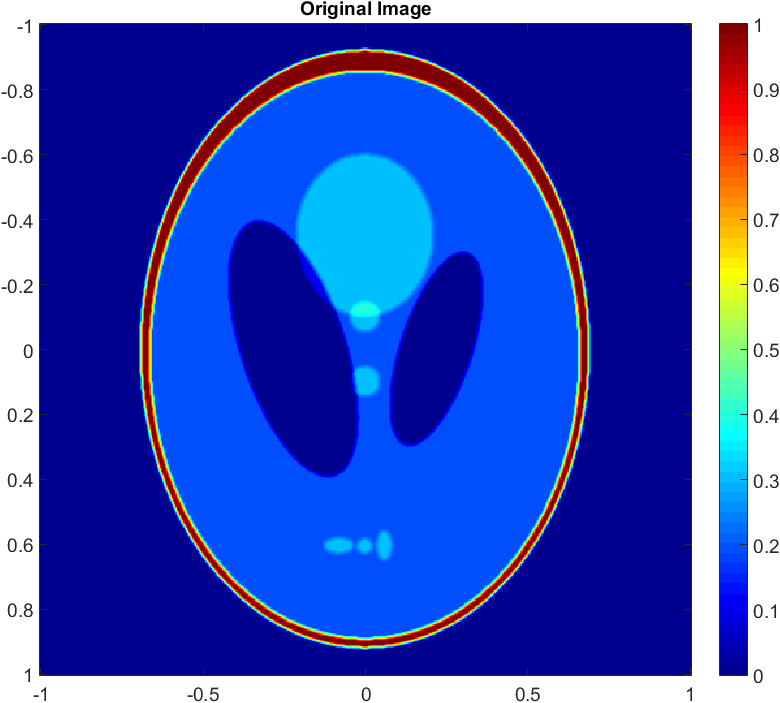}
\caption{\blue{From left to right: Ghost phantom; Shepp-Logan phantom.}}
\end{figure}
\begin{figure}[!ht]\center
\includegraphics[width=4.5cm]{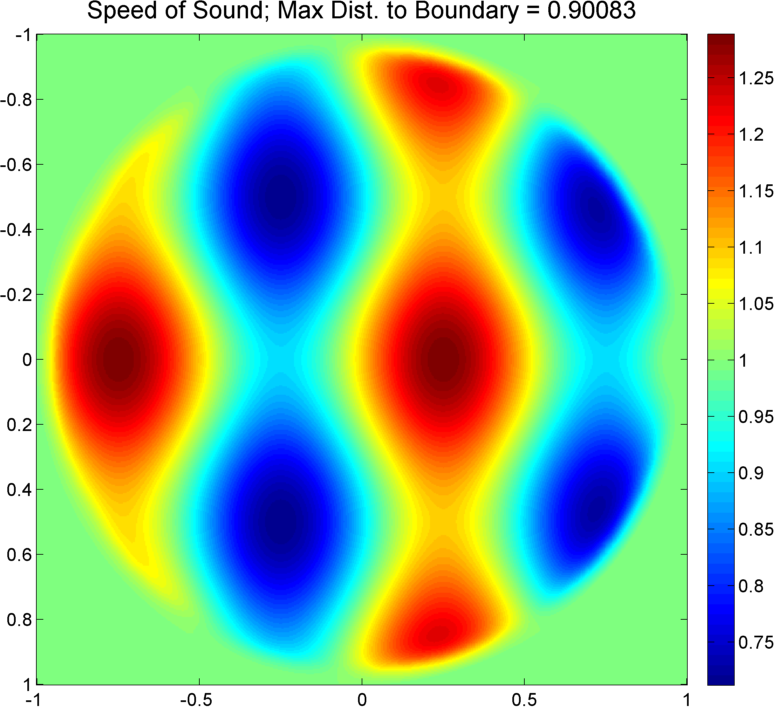}
\includegraphics[width=4.5cm]{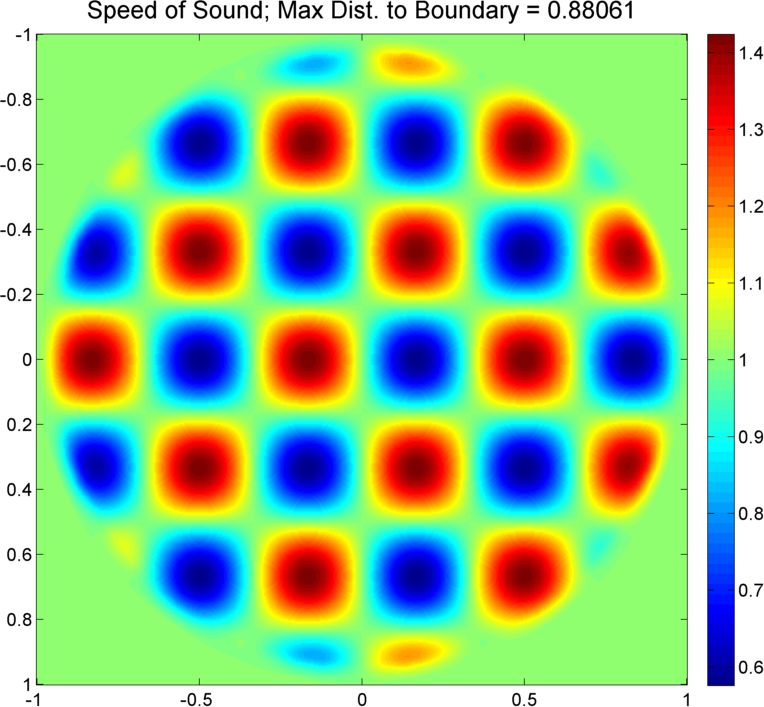}
\caption{\blue{From left to right: Non-trapping speed; Trapping speed.}}
\end{figure}

\begin{figure}[!ht]
\includegraphics[width=4.2cm]{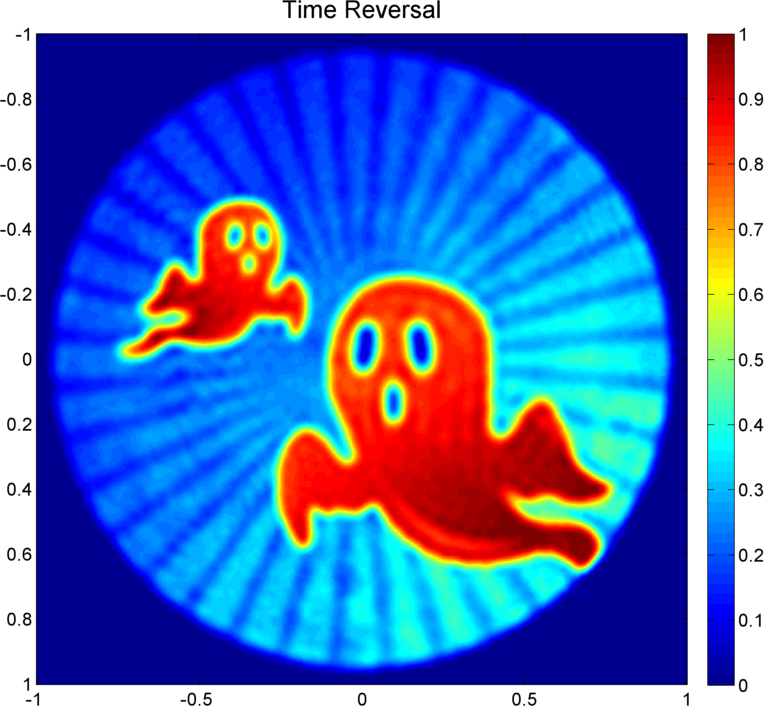}
\includegraphics[width=4.2cm]{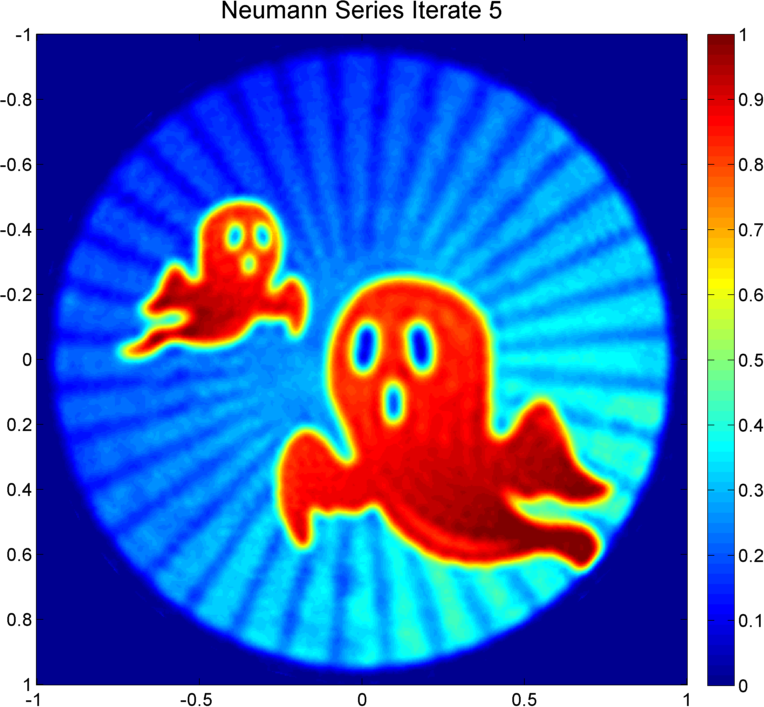}
\includegraphics[width=4.2cm]{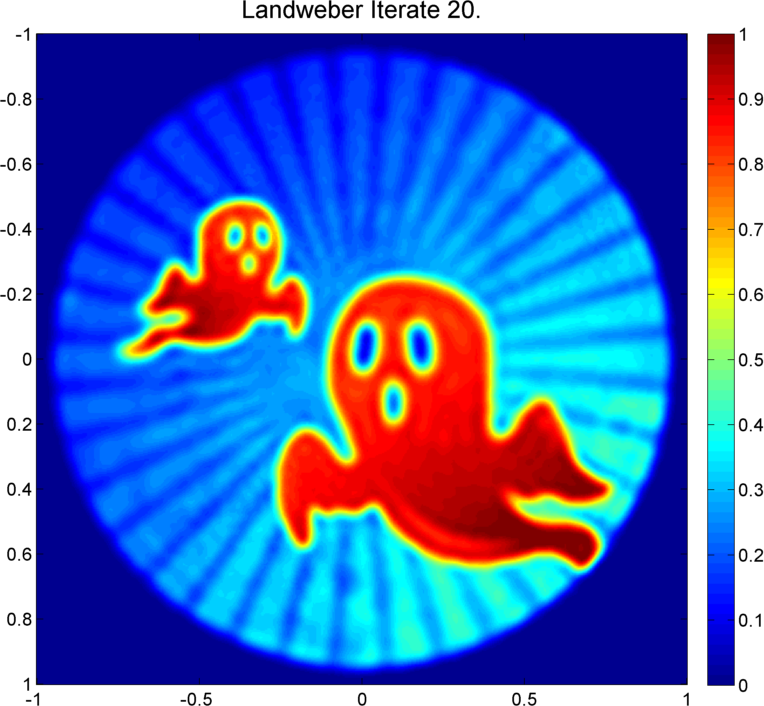}
\caption{Data: Non-trapping speed, $T=4T_0$. \blue{From left to right: Time reversal; 5th Neumann sum; Landweber iterate 20.}}
\label{fig:nontrap_long}
\end{figure}

\begin{figure}[!ht]
\includegraphics[width=4.2cm]{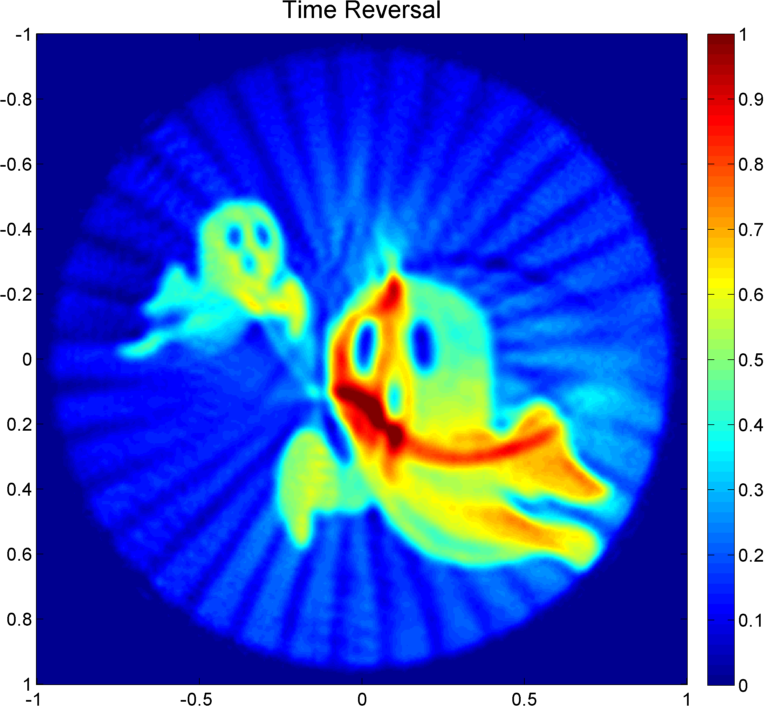}
\includegraphics[width=4.2cm]{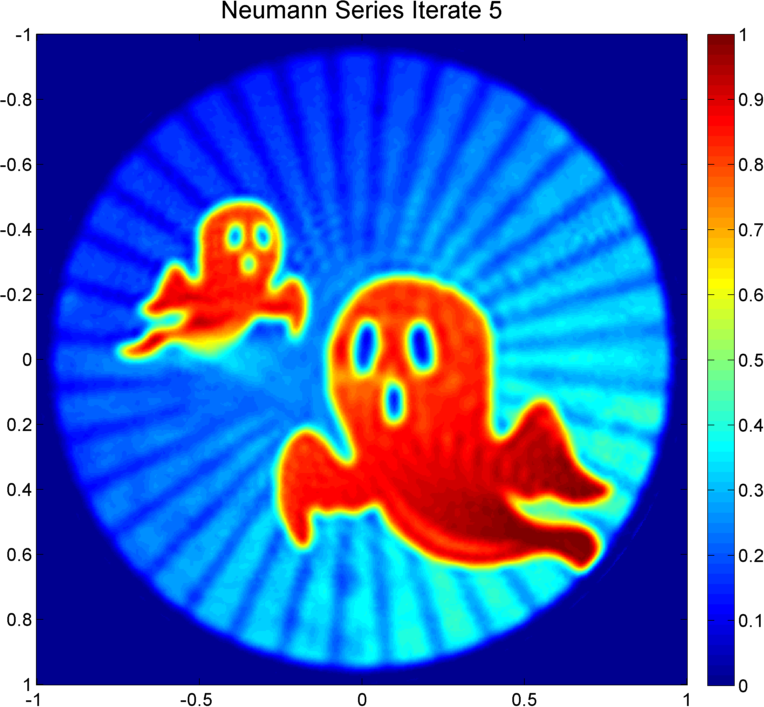}
\includegraphics[width=4.2cm]{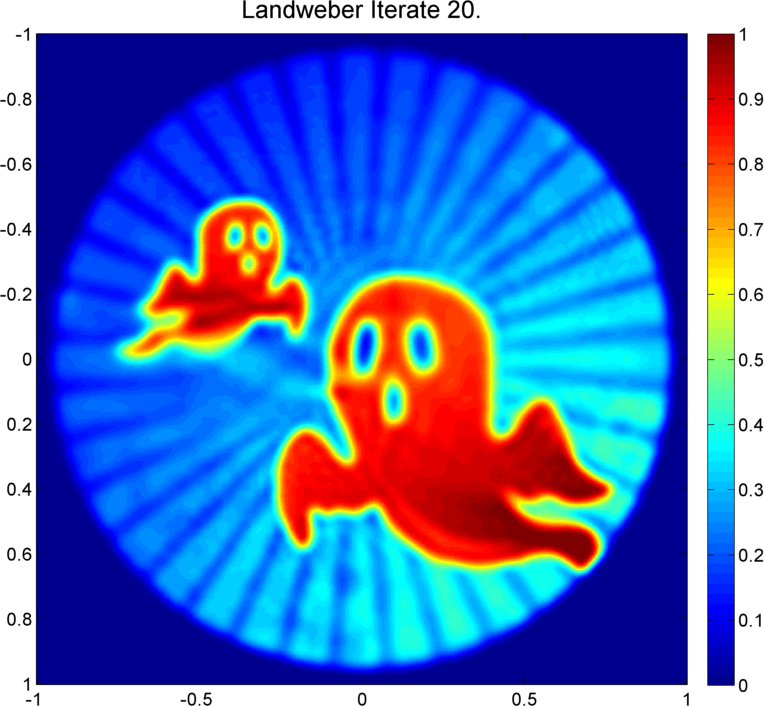}
\caption{Data: Non-trapping speed, $T=1.2T_0$. \blue{From left to right: Time reversal; 5th Neumann sum; Landweber iterate 20.}}
\label{fig:nontrap_short}
\end{figure}

\begin{figure}[!ht]
\includegraphics[width=4.2cm]{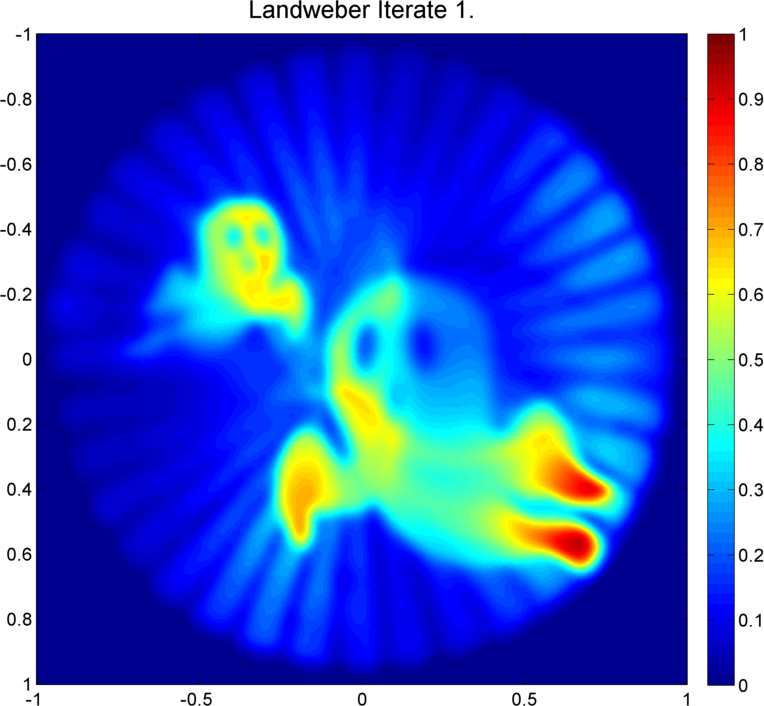}
\includegraphics[width=4.2cm]{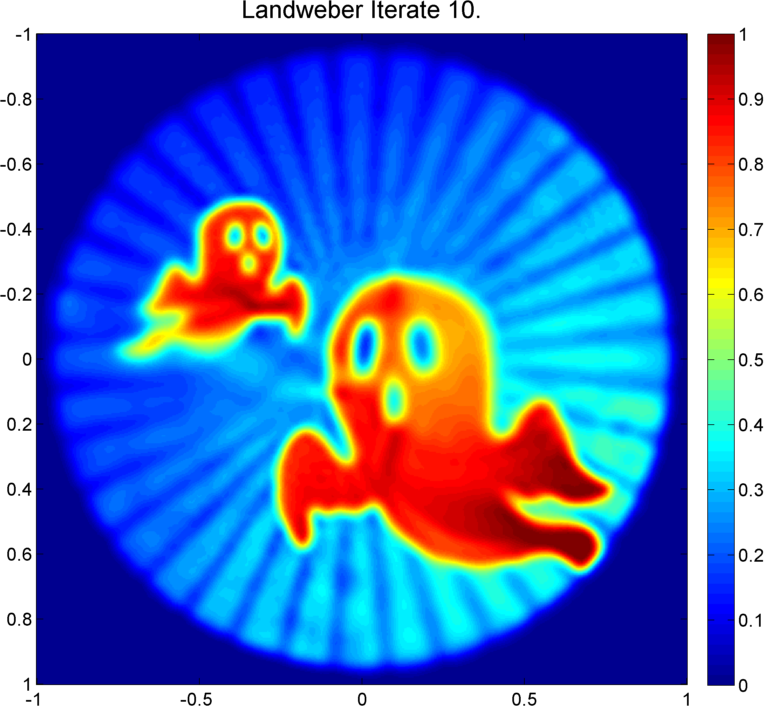}
\includegraphics[width=4.2cm]{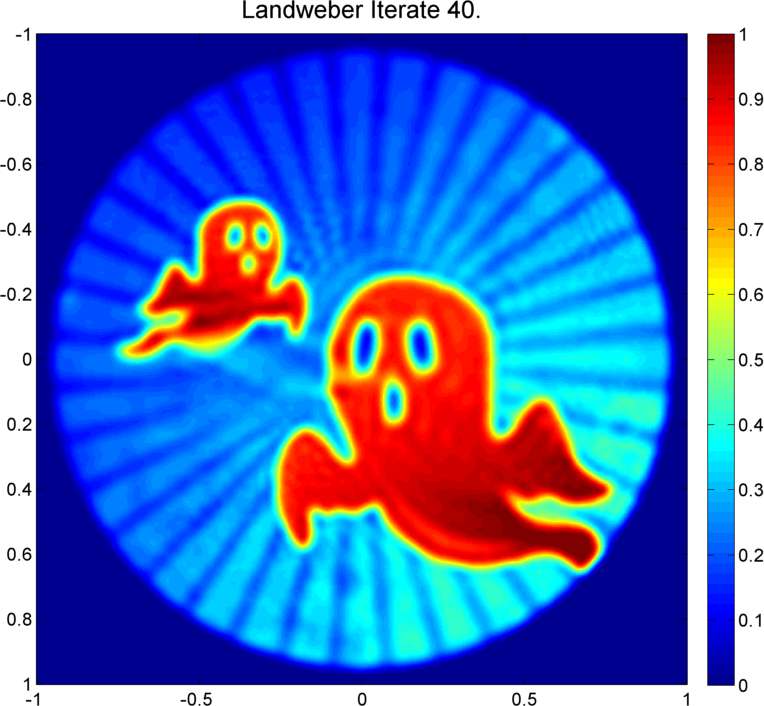}
\caption{Landweber reconstruction at different iterations. Data: Non-trapping speed, $T=1.2T_0$. \blue{From left to right: Iterate 1; Iterate 10; Iterate 50.}}
\label{fig:Landweber_iterations}
\end{figure}

\begin{figure}[!ht]
\includegraphics[width=4.2cm]{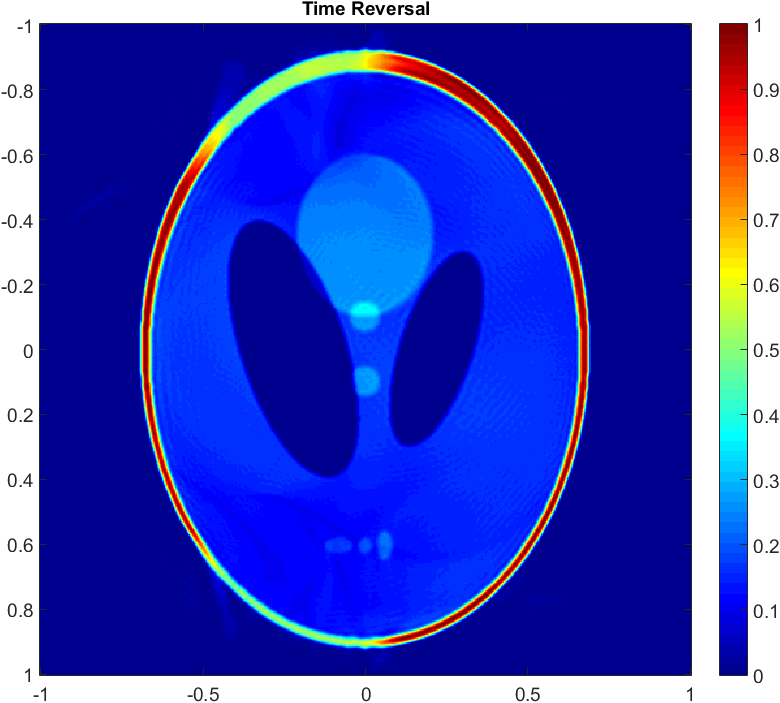}
\includegraphics[width=4.2cm]{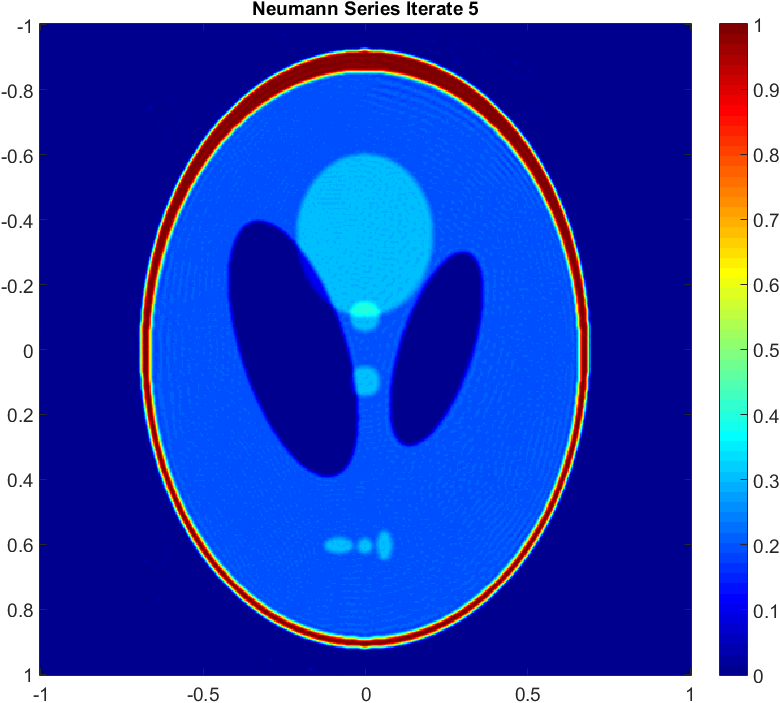}
\includegraphics[width=4.2cm]{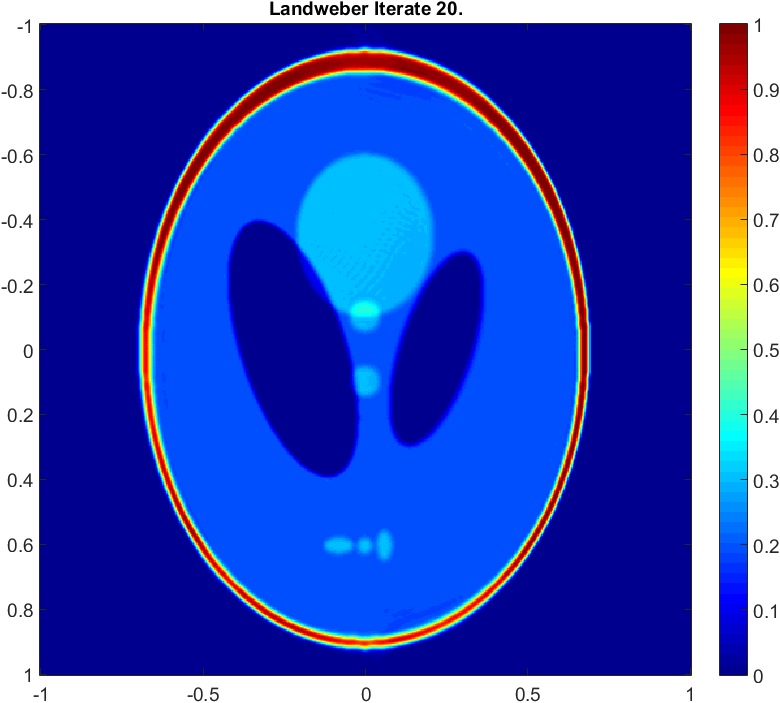}
\blue{\caption{Data: Non-trapping speed, $T=2T_0$. \blue{From left to right: Time reversal; 5th Neumann sum; Landweber iterate 20.}}}
\label{fig:nontrap_shepp}
\end{figure}
Let's first consider the case of the non-trapping speed 
\begin{align*}
c(x)=\left\{
\begin{aligned}
&1+0.2\sin(2\pi x_1)+0.1\cos(2\pi x_2),\quad x\in B_{R-\ve}\,,\\
&1\quad \text{outside }B_R\;.
\end{aligned}\right.
\end{align*}
Note that here and below, the speed of sound is smoothed out near the boundaries to satisfy our smoothness assumptions. 
The time reversal method already gives reliable reconstructions if the measurement time $T$ is sufficiently large. 
For the test example, a measurement time of $T=4T_0$ (see Remark \ref{rem:inj}) 
is enough to provide a quantitatively reasonable reconstruction by time reversal. 
This is illustrated in Figure \ref{fig:nontrap_long}.
We therefore are interested in the case where the measurement time is shorter, e.g. $T$ as near as possible at $T_0$. 
In Figure \ref{fig:nontrap_short} we compare the time reversal reconstruction and the Neumann series approach with our method, using $T=1.2T_0$.
The Landweber reconstruction is stopped at a suitable stage of iteration. In practice, the improvement is very large in the first steps, while it needs a lot of iterations to satisfy the discrepancy principle from Theorem \ref{thm:Landweber1}. 

The main differences are clearly visible to the naked eye: 
The time reversal reconstruction fails to compute the central point in the image and produces artifacts. These artifacts can be avoided by the use of the  harmonic extension as in \eqref{eq:TR}. We here present iterate $j=5$ of the Neumann series \eqref{eq:Neumann2}.
Also the Landweber reconstruction avoids to amplify the artifacts in the image center.
Moreover, it delivers the correct quantitative values of the initial pressure,
whereas time reversal underestimates these values. 
However, the smoothing step naturally included in every Landweber iteration seems to make 
this approach more stable than the time reversal Neumann series. 
In fact, at least with our method of numerical wave propagation, we have to stop the Neumann series after 
5-10 iterates before numerical errors are amplified too much. 

\begin{remark}[Convergence rates in practice]
The main difference between the convergence rates of the Neumann series and the Landweber iteration lies in the division by $c^2$ after every 
backpropagation step, indicated by \eqref{eq:adjoint}.

In Figure \ref{fig:Landweber_iterations} we show reconstructions using the non-trapping speed and measurement time $T=1.2T_0$. 
The measured error $\norm{y^\delta-Lf_k}_{L^2(\Sigma)}$ keeps decreasing till $k=50$.
However, the major visible improvements seem to occur within the first 20 iterations.
We therefore use the picture of iteration number 20 for our practical comparisons with the other methods.
\end{remark}

\subsection{Trapping sound speed}

\begin{figure}[!ht]
\includegraphics[width=4.2cm]{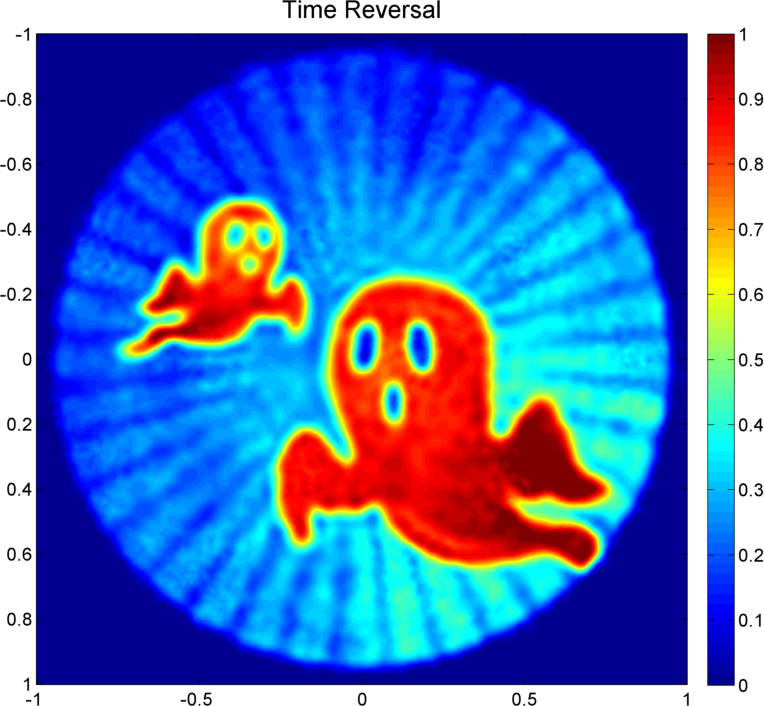}
\includegraphics[width=4.2cm]{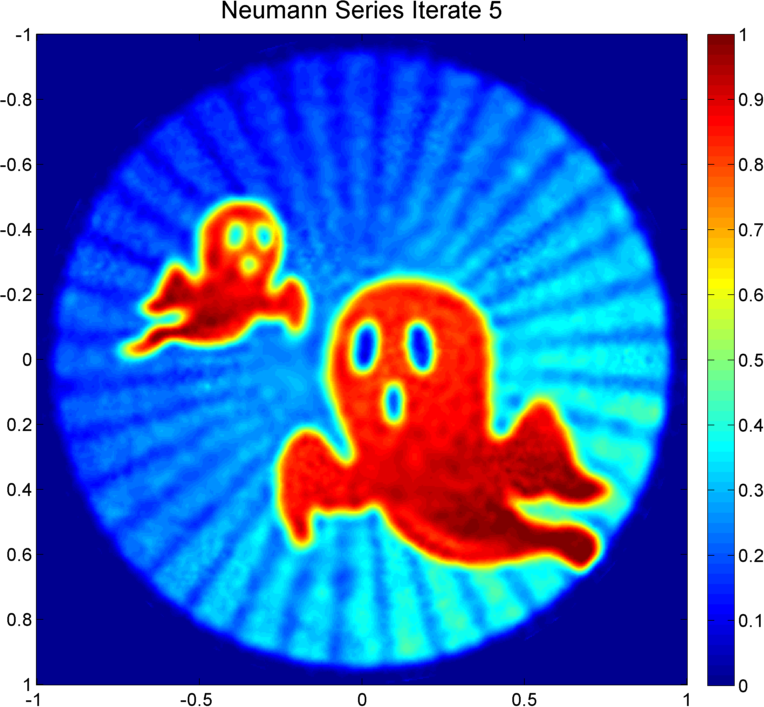}
\includegraphics[width=4.2cm]{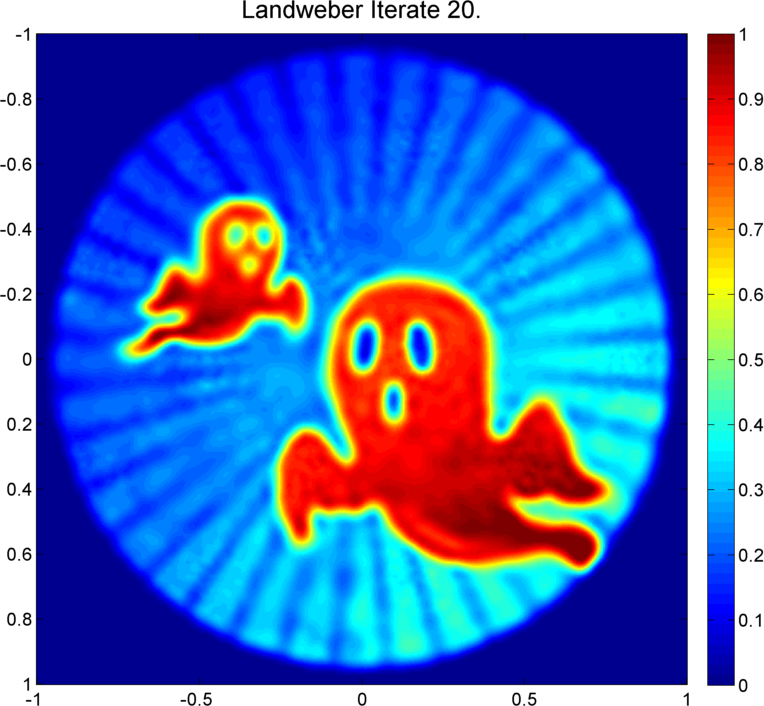}
\caption{Data: Trapping speed, $T=4T_0$. \blue{From left to right: Time reversal; 5th Neumann sum; Landweber iterate 20.}}
\label{fig:trap_long}
\end{figure}

\begin{figure}[!ht]
\includegraphics[width=4.2cm]{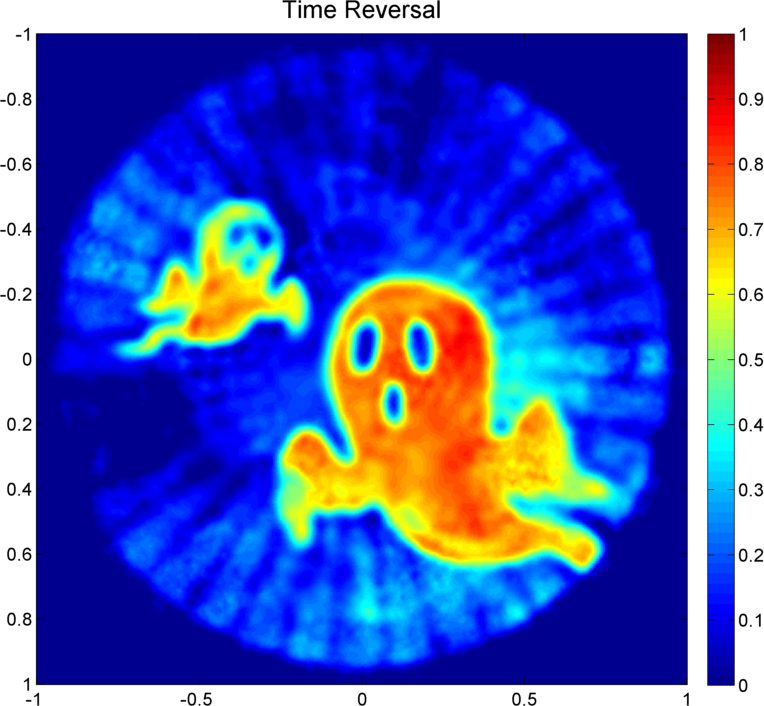}
\includegraphics[width=4.2cm]{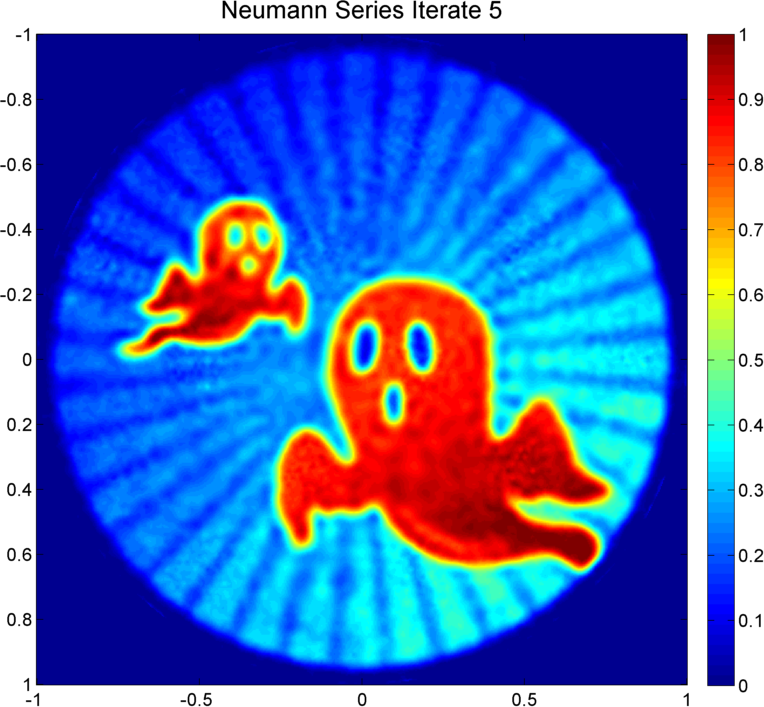}
\includegraphics[width=4.2cm]{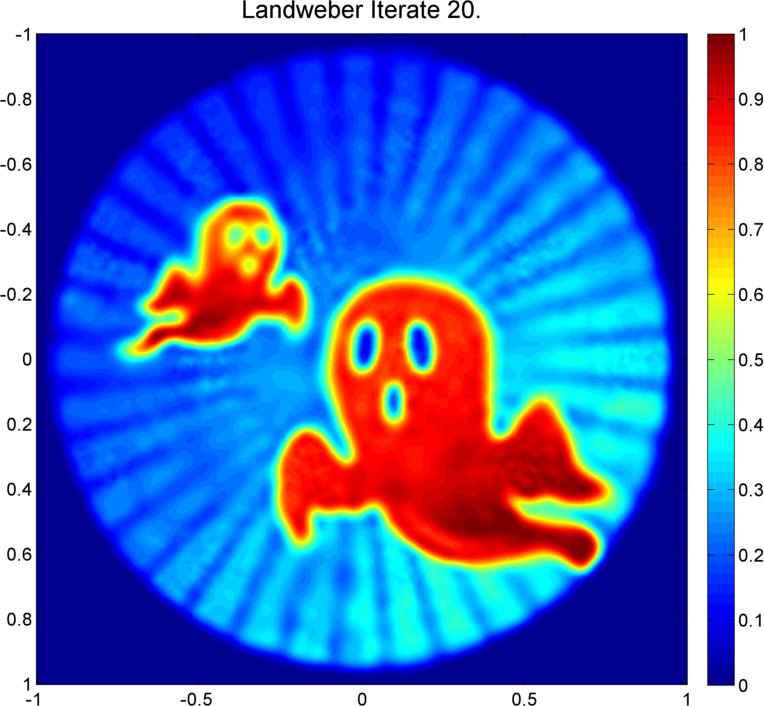}
\caption{Data: Trapping speed, $T=2T_0$. \blue{From left to right: Time reversal; 5th Neumann sum; Landweber iterate 20.}}
\label{fig:trap_mid}
\end{figure}

\begin{figure}[!ht]
\includegraphics[width=4.2cm]{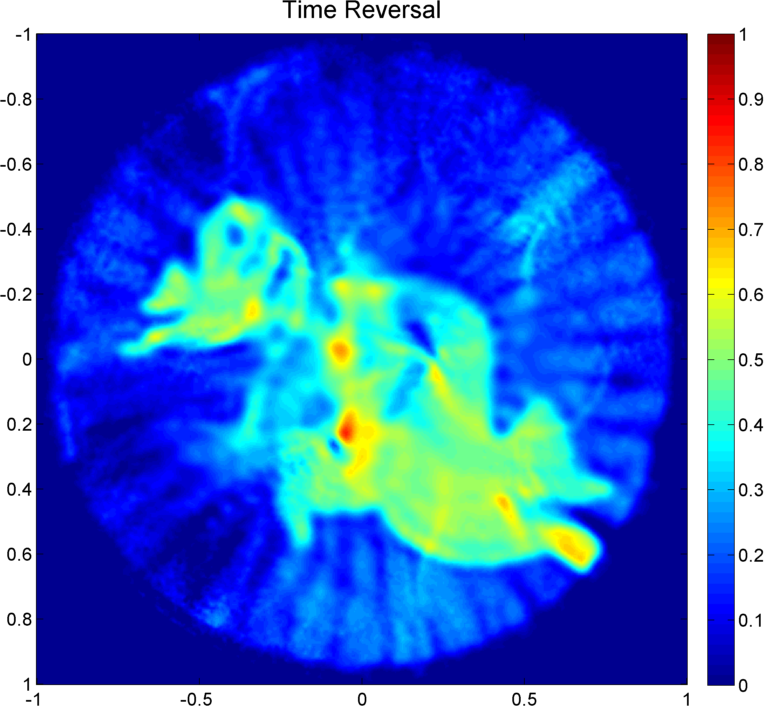}
\includegraphics[width=4.2cm]{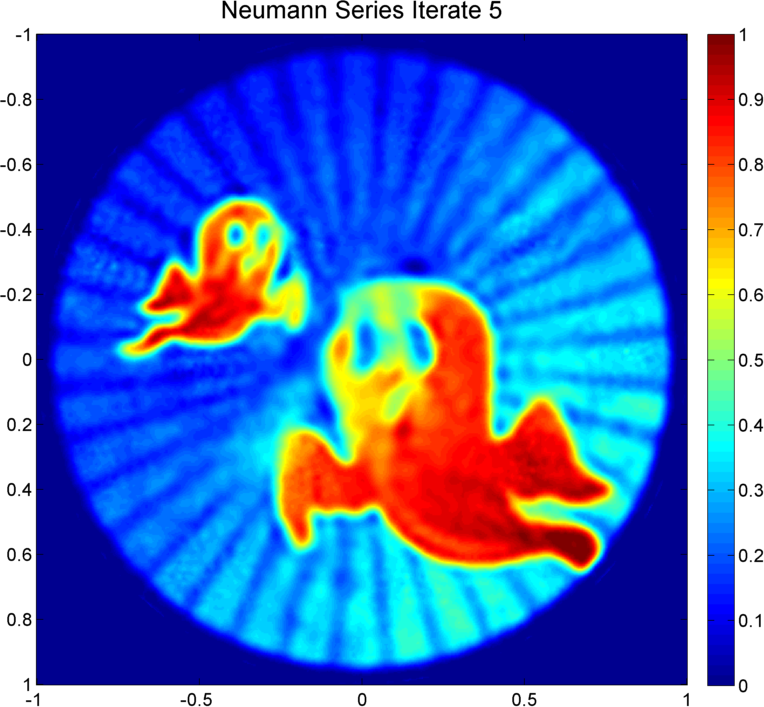}
\includegraphics[width=4.2cm]{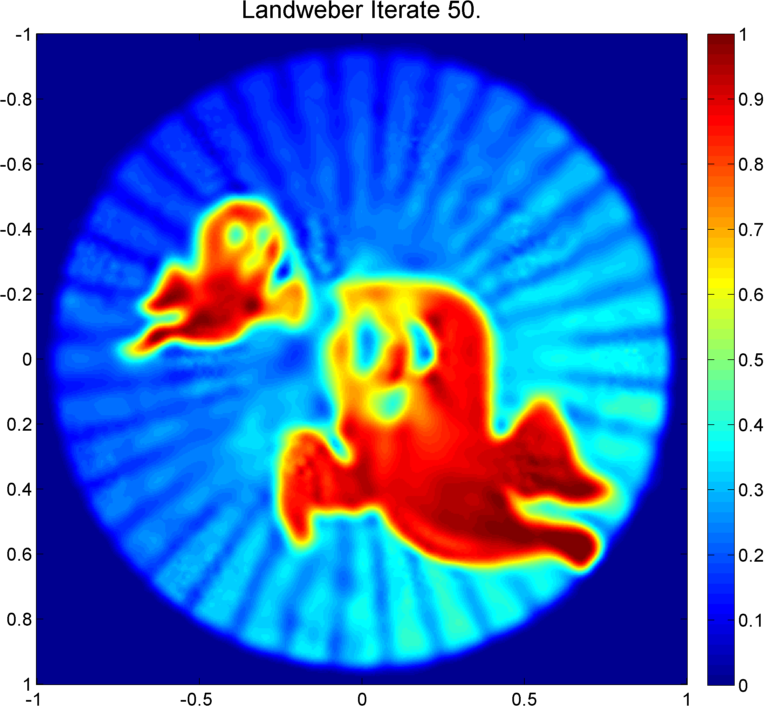}
\caption{Data: Trapping speed, $T=1.2T_0$. \blue{From left to right: Time reversal; 5th Neumann sum; Landweber iterate 50.}}
\label{fig:trap_short}
\end{figure}

In the second example, we want to deal with the trapping speed
\begin{align*}
c(x)=\left\{
\begin{aligned}
&1+0.5\sin(-3\pi x_1)\cos(3\pi x),\quad x\in B_{R-\ve}\,,\\
&1\quad \text{outside }B_R\;.
\end{aligned}\right.
\end{align*}
In this case, there are geodesics present that do not propagate singularities to the measurement surface within finite time. 
The Landweber approach is the only one that gives a theoretical convergence result in this case.
In practice, we see that conventional time reversal, at least for $T=1.2T_0$, fails to give a detailed reconstruction. The Neumann series approach and the Landweber iteration behave similarly, again with the advantage of the Neumann series giving faster convergence, whereas the Landweber gives a regularized solution that seems to be more robust against numerical errors and noise (figures \ref{fig:trap_long}, \ref{fig:trap_mid} and \ref{fig:trap_short}).

\blue{\section{Conclusions}
In this paper, an infinite dimensional Hilbert space regularization framework is established for solving the photoacoustic inverse problem with variable speed of sound. 
The introduced Landweber approach iteratively solves the least squares problem, 
a standard stopping criterion serves as regularization parameter. 
By regularization theory, convergence of the method is guaranteed. 
The reconstruction method is compared to the common state-of-the-art reconstruction techniques in the field, which are based on solving a time reversed wave equation.
It thereby proves to be a flexible and feasible alternative for computational photoacoustic tomography with variable speed.
}

\section*{Acknowledgement}
The work of TG and OS is supported by the Austrian Science Fund (FWF), Project P26687-N25 
Interdisciplinary Coupled Physics Imaging. 
\blue{The authors are grateful to the Erwin Schr\"odinger institute 
for warm hosting. Thanks go to the referees for the helpful comments and
to Roland Herzog for remarks on the initial version.
Special thanks go to Silvia Falletta for providing valuable hints concerning
details to discretization and implementation of the coupled system.}  

\def\cprime{$'$} \providecommand{\noopsort}[1]{}

\appendixpage
\appendix
\section{Well-posedness of the transmission problem}
\begin{notation}
\begin{itemize}
\item The space
      \begin{equation} \label{eq:V}
      V:= \set{ u \in H_0^1(\RR^n): u|_\Omega \in H^1(\Omega)}
      \end{equation}
      is a Hilbert space with respect to the inner product
      \begin{equation*}
       \inner{\phi_1}{\phi_2}_V = \int_\Omega \phi_1(x) \phi_2(x)\,d x + \int_{\RR^n} \nabla \phi_1 \cdot \nabla \phi_2\,dx\;.
      \end{equation*}
      The completeness follows directly from the properties of the Sobolev spaces \\
      $H_0^1(\RR^n)$ and $H^1(\Omega)$. 
      
      The dual space is denoted by $V'$.
\item \blue{Let $\mathcal{S} \subseteq \RR^n$, then 
      $C_0^\infty(0,T;\mathcal{S})$ denotes the space of functions $v \in C^\infty((0,T) \times \mathcal{S})$
      with compact support in $(0,T)$ with respect to $t$.}
\end{itemize}
\end{notation}
\label{sec:wellposedness}
\begin{definition}
For $h \in C^\infty((0,T) \times \partial \Omega)$. A weak solution $z$ of \eqref{eq:wave_adj} satisfies: 
\begin{itemize}
\item 
\begin{equation*}
  z \in L^2(0,T;V)\,,
  z' \in L^2(0,T;L^2(\RR^n))\,,
  z'' \in L^2(0,T;V'),
\end{equation*}
where $V$ is defined in \eqref{eq:V}, together with
\item $z(T)=z'(T)=0$. Note that \blue{$V \subseteq L^2(\RR^n) \subseteq V'$} and thus 
      $z \in H^1(0,T;L^2(\RR^n))$ and $z'\in H^1(0,T;V')$, such that the traces make sense.
      In fact $z(T) \in L^2(\RR^n)$ and $z'(T) \in V'$.
\item      
      \begin{equation}
      \label{eq:weak_adjoint}
       \begin{aligned}
        & \int_0^T \int_{\RR^n} \frac{1}{c^2} z''v\,dx\,dt\; + \;
        \int_0^T \int_{\RR^n} \nabla z(t) \cdot \nabla v(t)\,dx\,dt\; \\
        =& - \int_0^T \int_{\partial \Omega} h(t) v(t)\,d S(x)\,dt\;\,,
        \qquad \text{ for all } v \in L^2(0,T;V).
       \end{aligned}
       \end{equation}
       Note, that the definition of a weak solution is not used consistently in the literature. Evans \cite{Eva98}, for instance, uses 
       the {test} functions $v$ time independently, and formulates a weak form of \eqref{eq:wave_adj} for almost all $t \in (0,T)$.
\end{itemize}
\end{definition}
\begin{theorem}
For $h \in C^\infty((0,T) \times \partial \Omega)$ there exists a weak solution $z$ of \eqref{eq:wave_adj} - that is of \eqref{eq:weak_adjoint}.
\end{theorem}
\begin{proof}
The proof is similar to \cite[Section 7.2.2]{Eva98}.
\begin{enumerate}
 \item First we construct a Galerkin-approximation:
       Let $\{w_k\}_{k=1}^\infty$ be an orthonormal basis of $L^2(\RR^n)$, which is an orthogonal basis for $H_0^1(\RR^n)$.
       For fixed $m$ let 
       \begin{equation*}
         z_m(t)=\sum_{k=1}^m d_m^k(t) w_k\,,\qquad t \in (0,T)\;,
       \end{equation*}
       where 
       \begin{equation}
       \label{eq:init}
        d_m^k(T)= 0\,, d_m^k{}'(T)= 0 \text{ meaning that } z_m(T)=z_m'(T)=0\;.
       \end{equation}
       and 
       \begin{equation}\label{eq:Galerkin}
       \begin{aligned}
        ~ & \inner{\frac{z_m''}{c^2}}{w_k}_{L^2(\RR^n)} + \inner{z_m}{w_k}_{H_0^1(\RR^n)} = \inner{h}{w_k}_{L^2(\Sigma)}\,, \\
          & \text{ for all } 0< t < T\,, \text{ for all } k=1,\ldots,m\;.
       \end{aligned}
       \end{equation}
       Analogously to \cite[Theorem 1, Section 7.2.2]{Eva98} the Galerkin approximation can be shown.
 \item The following estimates are different to \cite[Theorem 1, Section 7.2.2]{Eva98} and thus included - it is essential 
       to consider an additional time integration. 
       As in \cite{Eva98} we use $z_m'(t)=\sum_{k=1}^m d_m^k{}'(t) w_k$ as a test function in the Galerkin approximation. 
       However, we do not apply it pointwise for every $\tau \in (0,T)$, but in integrated form. 
       Then from \eqref{eq:Galerkin} it follows that 
       \begin{equation*}
        \begin{aligned}
         \int_\tau^T\int_{\RR^n}\frac{z_m''(\hat{t})}{c^2} z_m'(\hat{t})+\nabla z_m\cdot\nabla z_m'dxd\hat{t} &=-\int_{\d\Omega}\int_\tau^T h(\hat{t})z_m'(\hat{t})d\hat{t}d S(x)\\
         & \text{ for all } \tau \in (0,T)\;.
        \end{aligned}
       \end{equation*}
       Partial integration of the right hand side and evaluation of the integral terms on the left hand gives
       \begin{equation}
       \label{eq:basic}
        \begin{aligned}
        & \int_\tau^T\frac{d}{d\hat{t}}\left(\norm{\frac{z_m'(\hat{t})}{c^2}}^2_{L^2(\RR^n)}+\norm{z_m(\hat{t})}_{H_0^1(\RR^n)}^2\right) d\hat{t}\\
        = & 2\left(\int_{\d\Omega}\int_\tau^T h'(\hat{t})z_m(\hat{t})\,d\hat{t} dS(x) - \int_{\partial \Omega} h(\tau)z_m(\tau)dS(x)\right)\;.
        \end{aligned}
       \end{equation}
       Since $z_m(T)=z_m'(T)=0$, the left hand side equals 
       \begin{equation*}
        -(\norm{z_m'(\tau)}^2_{L^2(\RR^n)}+\norm{z_m(\tau)}_{H_0^1(\RR^n)}^2)\;.
       \end{equation*}
       Estimating the right hand side by Cauchy-Schwarz-inequality and using mean inequality, 
       we get for an arbitrary $D >0$:
       \begin{equation*}
        \begin{aligned}
        & \norm{\frac{z_m'(\tau)}{c^2}}^2_{L^2(\RR^n)}+\norm{z_m(\tau)}_{H_0^1(\RR^n)}^2\\
        \leq &\frac{1}{D^2}\left(\int_\tau^T \norm{h'(\hat{t})}_{L^2(\d\Omega)}^2d\hat{t}+\norm{h(\tau)}_{L^2(\d\Omega)}^2\right)\\ 
        &\quad + D^2\left(\int_\tau^T\norm{z_m(\hat{t})}_{L^2(\d\Omega)}^2 d\hat{t}+\norm{z_m(\tau)}_{L^2(\d\Omega)}^2\right)\;.
        \end{aligned}
       \end{equation*}
       Let 
       \begin{equation*}
       \begin{aligned}
          \mathcal{C}(h,\tau) &:= \int_\tau^T \norm{h'(\hat{t})}_{L^2(\d\Omega)}^2d\hat{t}+\norm{h(\tau)}_{L^2(\d\Omega)}^2\,,\\
          \hat{\mathcal{C}}(h,t) &:= \int_t^T \mathcal{C}(h,\tau) \tau\,,\\
       \end{aligned}
       \end{equation*}
       and using the trace theorem \eqref{eq:norm} it follows that 
       \begin{equation*}
        \begin{aligned}
        & \norm{\frac{z_m'(\tau)}{c^2}}^2_{L^2(\RR^n)}+\norm{z_m(\tau)}_{H_0^1(\RR^n)}^2\\
        \leq & \frac{1}{D^2} \mathcal{C}(h,\tau)
        + D^2C_\gamma^2\left(\int_\tau^T\norm{z_m(\hat{t})}_{H^1(\Omega)}^2 d\hat{t}+\norm{z_m(\tau)}_{H^1(\Omega)}^2\right)\;.
        \end{aligned}
       \end{equation*}
       By integrating of $\tau$ over $[t,T]$ we get the following estimate for all $t \in (0,T)$:
       \begin{equation}
       \label{eq:energy}
        \begin{aligned}
        & \int_t^T \norm{\frac{z_m'(\tau)}{c^2}}^2_{L^2(\RR^n)}+\norm{z_m(\tau)}_{H_0^1(\RR^n)}^2 d\tau\\
        \leq&  \frac{1}{D^2} \hat{\mathcal{C}}(h,t)
        + D^2C_\gamma^2\int_t^T \left(\int_\tau^T\norm{z_m(\hat{t})}_{H^1(\Omega)}^2 d\hat{t}+\norm{z_m(\tau)}_{H^1(\Omega)}^2\right)\,d\tau\;.
        \end{aligned}
       \end{equation}
       In addition, because $z_m(T)=0$, we know that
       \begin{equation}
       \label{eq:en2}
       \begin{aligned}
       \left(\int_{\Omega} z_m(\tau)dx\right)^2 &= \left(\int_{\Omega} z_m(\tau)-z_m(T)dx\right)^2 \\
       &= \left( \int_\tau^T \int_{\Omega} z_m'(\hat{t})\,d x d\hat{t} \right)^2\\
       &\leq T \blue{|\Omega| c_\text{max}^4 }
       \int_\tau^T \int_{\Omega} \left(\frac{z_m'(\hat{t})}{c^2}\right)^2\,d x d\hat{t}\\ 
       &= T \blue{|\Omega|c_\text{max}^4}   \int_\tau^T \norm{\frac{z_m'(\hat{t})}{c^2}}^2_{L^2(\RR^n)}d\hat{t}\;.
       \end{aligned}
       \end{equation}
       Moreover, we know from the inequality from \blue{\cite[5.8, Theorem 1]{Eva98}} that 
       \begin{equation*}
        \norm{z_m(\tau) - \blue{\frac1{|\Omega|}}\int_{\Omega} z_m(\tau)dx}_{L^2(\Omega)} \leq C_G \norm{z_m(\tau)}_{H_0^1(\RR^n)}\,,
       \end{equation*}
       which implies that for all $\tau \in (0,T)$:
       \begin{equation*}
        \norm{z_m(\tau)}_{L^2(\Omega)}^2 \leq \blue{\frac2{|\Omega|}} \left(\int_{\Omega} z_m(\tau)dx\right)^2 + 2C_G^2 \norm{z_m(\tau)}_{H_0^1(\RR^n)}^2\,,
       \end{equation*}
       and consequently
       \begin{equation*}
        \norm{z_m(\tau)}_{H^1(\Omega)}^2 \leq \blue{\frac2{|\Omega|}} \left(\int_{\Omega} z_m(\tau)dx\right)^2 + (2C_G^2+1) \norm{z_m(\tau)}_{H_0^1(\RR^n)}^2\;.
       \end{equation*}
       Thus it follows from \eqref{eq:energy} that
       \begin{equation}
       \label{eq:energy1}
        \begin{aligned}
        & \int_t^T \norm{\frac{z_m'(\tau)}{c^2}}^2_{L^2(\RR^n)}+\norm{z_m(\tau)}_{H_0^1(\RR^n)}^2 d\tau\\
        \leq&  \frac{1}{D^2} \hat{\mathcal{C}}(h,t) \\
        &
        + D^2C_\gamma^2(2C_G+1)\int_t^T \left(\int_\tau^T\norm{z_m(\hat{t})}_{H_0^1(\RR^n)}^2 d\hat{t} +\norm{z_m(\tau)}_{H_0^1(\RR^n)}^2\right)\,d\tau \\
        & \qquad + \blue{\frac{2D^2C_\gamma^2 (T+1)}{|\Omega|}} \int_0^T \left(\int_{\Omega} z_m(\tau)dx\right)^2 d\tau\;.
        \end{aligned}
       \end{equation}
       Taking $t=0$ then gives 
       \begin{equation}
       \label{eq:energy1a}
        \begin{aligned}
        & \int_0^T \norm{\frac{z_m'(\tau)}{c^2}}^2_{L^2(\RR^n)}+\norm{z_m(\tau)}_{H_0^1(\RR^n)}^2 d\tau\\
        \leq&  \frac{1}{D^2} \hat{\mathcal{C}}(h,t)\\
        & + D^2C_\gamma^2(2C_G+1) (T+1) \int_0^T\norm{z_m(\hat{t})}_{H_0^1(\RR^n)}^2 d\hat{t}\\
        & \qquad + \blue{\frac{2D^2C_\gamma^2 (T+1)}{|\Omega|}} \int_0^T \left(\int_{\Omega} z_m(\tau)dx\right)^2 d\tau\;.
        \end{aligned}
       \end{equation}
      Thus it follows from \eqref{eq:en2} that 
       \begin{equation*}
        \begin{aligned}
          ~ & \frac{1}{2 T\blue{|\Omega|  c_\text{max}^4}} \int_0^T \left(\int_{\Omega} z_m(\tau)dx\right)^2\,d\tau + \\
            & \quad
              \frac{1}{2} \int_0^T \norm{\frac{z_m'(\tau)}{c^2}}^2_{L^2(\RR^n)}+\norm{z_m(\tau)}_{H_0^1(\RR^n)}^2 d\tau\\
        \leq& \frac{1}{D^2} \hat{\mathcal{C}}(h,t)\\
            &   + D^2C_\gamma^2(2C_G+1) (T+1) \int_0^T\norm{z_m(\hat{t})}_{H_0^1(\RR^n)}^2 d\hat{t}\\
        & \qquad + \blue{\frac{2D^2C_\gamma^2 (T+1)}{|\Omega|}} \int_0^T \left(\int_{\Omega} z_m(\tau)dx\right)^2 d\tau\;.
        \end{aligned}
       \end{equation*}
        Choosing $D$ such that 
        \begin{equation*}
         \max \set{ D^2 C_\gamma^2 (2C_G+1) (T+1), 2\blue{|\Omega|}D^2C_\gamma^2 (T+1)} \leq \frac{1}{2}\,,
        \end{equation*}
        provides that 
        \begin{equation}
       \label{eq:energy2}
        \begin{aligned}
          ~ & \frac{1}{4 T\blue{|\Omega|  c_\text{max}^4}} \int_0^T \left(\int_{\Omega} z_m(\tau)dx\right)^2\,d\tau + \\
            & \frac{1}{4} \int_0^T \norm{\frac{z_m'(\tau)}{c^2}}^2_{L^2(\RR^n)}+\norm{z_m(\tau)}_{H_0^1(\RR^n)}^2 d\tau\\
        \leq& \frac{1}{D^2} \hat{\mathcal{C}}(h,t)\;.
        \end{aligned}
       \end{equation}
       Therefore there exists a constant $\mathcal{C}$ such that 
       \begin{equation}
        \label{eq:uniform}
        \max \set{ \norm{z_m'}_{L^2(0,T;\RR^n)}, 
                   \norm{z_m}_{L^2(0,T;H_0^1(\RR^n))},
                   \norm{z_m|_\Omega}_{L^2(0,T;H^1(\Omega))}
           } \leq \mathcal{C}\,,
       \end{equation}
       where the last estimate is again due to inequality \blue{\cite[5.8, Theorem 1]{Eva98}}, already used above.

       Because $\set{w_k}$ is a basis it follows from \eqref{eq:Galerkin} that for all $v \in L^2(0,T;V)$:
       \begin{equation*}
        \sup_v \inner{\frac{z_m''}{c^2}}{v}_{L^2(0,T;\RR^n)} 
        \leq 
        \sup_v \inner{z_m}{v}_{H_0^1(0,T;\RR^n)} + \sup_v \inner{h}{v}_{L^2(\Sigma)} \leq C(h,\mathcal{C})\,,
       \end{equation*}
       where $C(h,\mathcal{C})$ is constant depending on $h$ and $\mathcal{C}$, but is not dependent on $m$.
       This shows that $\blue{z_m''} \in L^2(0,T;V')$.
 \item \eqref{eq:energy2} guarantees that $\set{z_m}$ is uniformly bounded in $L^2(0,T;H_0^1(\RR^n))$,
       and\\ $\set{z_m|_\Omega}$ is uniformly bounded in $L^2(0,T;H^1(\Omega))$.
       Moreover, $\set{z_m'}$ is uniformly bounded in $L^2(0,T;L^2(\RR^n))$, and 
       $\set{z_m''}$ is uniformly bounded in $L^2(0,T;V')$.
       Thus it has a weakly convergent subsequence, which we again denote by $\set{z_m}$ which is 
       converging weakly in the following sense:
       \begin{equation*}
        \begin{aligned}
         \set{z_m} \rightharpoonup z  \text{ with respect to } L^2(0,T;H_0^1(\RR^n))\,,\\
         \set{z_m|_\Omega} \rightharpoonup z|_\Omega  \text{ with respect to } L^2(0,T;H^1(\Omega))\,,\\
         \set{z_m'} \rightharpoonup \psi  \text{ with respect to } L^2(0,T;L^2(\RR^n))\,,\\
         \set{z_m''} \rightharpoonup \rho  \text{ with respect to } L^2(0,T;V')\;.\\
        \end{aligned}
       \end{equation*}
       We note that the trace operator $\gamma:H^1(\Omega) \to L^2(\partial \Omega)$ is bounded. Therefore, every 
       test function $v \in L^2(0,T;V)$ satisfies $v|_{\partial \Omega} \in L^2(0,T;L^2(\partial \Omega))$.
       Consequently,
       \begin{equation*}
        \int_0^T \int_{\partial \Omega} z_m v\,dS(x),dt  \to \int_0^T \int_{\partial \Omega} z v\,dS(x),dt\;.
       \end{equation*}
       Now, let $v \in C_0^\infty((0,T) \times \RR^n)$, then $v', v'' \in C_0^\infty((0,T) \times \RR^n)$ as well, such that 
       \begin{equation*}
       \begin{aligned}
        \int_0^T \int_\Omega \psi v dx dt &= \lim \int_0^T \int_\Omega z_m' v dx dt \\
        &= - \lim \int_0^T \int_\Omega z_m v' dx dt \\
        &= \lim \int_0^T \int_\Omega z v' dx dt \\
        &= \lim \int_0^T \int_\Omega z' v dx dt\,,
       \end{aligned}
       \end{equation*}
       which implies that $\psi = z'$. Analogously one can show that $z_m''\to z''$ in $L^2(0,T;V')$.
       Therefore, $z$ it is a solution of \eqref{eq:weak_adjoint}.
\end{enumerate}
\end{proof}
\begin{theorem}
 The solution of \eqref{eq:wave_adj} is unique in $L^2(0,T;V)$.
\end{theorem}
\begin{proof}
Let us assume that there exist two solutions $z_1,z_2$ of \eqref{eq:wave_adj}.
Then 
 \begin{equation}\label{eq:wave_adj_diff}
 \begin{aligned}
    \frac1{c^2}\, (z_1-z_2)'' - \Delta (z_1-z_2) &=0 \text{ in } \RR^n \backslash \partial \Omega \times (0,T)\,,\\
          (z_1-z_2)(T) = (z_1-z_2)'(T) &=0  \text{ in } \RR^n, \\
          \left[ (z_1-z_2) \right] =0\,, \quad
        \left[  \frac{\partial (z_1-z_2)}{\partial \bm n} \right] &= 0 \text{ on } \partial \Omega \times (0,T)\;.
\end{aligned}
\end{equation}
But this solution has only a trivial solution, which one sees from \eqref{eq:E} and by noting the boundary conditions at $\infty$ 
of an $H_0^1(\RR^n)$ function.
\end{proof}

\end{document}